\documentclass[11pt]{amsart}
\usepackage{amsmath,amssymb,color,cite,verbatim,mathtools,enumitem,placeins,multirow}
\usepackage{etoolbox}
\apptocmd{\sloppy}{\hbadness 10000\relax}{}{}

\numberwithin{equation}{section}

\newtheorem{thm}[equation]{Theorem}
\newtheorem{prop}[equation]{Proposition}
\newtheorem{lemma}[equation]{Lemma}
\newtheorem{cor}[equation]{Corollary}

\theoremstyle{definition}
\newtheorem{rmk}[equation]{Remark}
\newtheorem{notation}[equation]{Notation}
\newtheorem{defn}[equation]{Definition}

\newcommand{\F}{\mathbb{F}}
\newcommand{\bP}{\mathbb{P}}

\newcommand{\Z}{\mathbb{Z}}

\DeclareMathOperator{\ord}{ord}
\DeclareMathOperator{\lcm}{lcm}

\DeclareMathOperator{\PSL}{PSL}

\DeclareMathOperator{\Tr}{Tr}

\newcommand{\mybar}[1]{#1\llap{$\overline{\phantom{\rm#1}}$}}
\newcommand{\abs}[1]{\lvert #1 \rvert}

\usepackage[colorlinks,pagebackref,pdftex, bookmarks=false]{hyperref}

\begin{document}

\title[Permutation quadrinomials]{Determination of a class of permutation quadrinomials}

\author{Zhiguo Ding}
\address{
  Hunan Institute of Traffic Engineering,
  Hengyang, Hunan 421001 China
}
\email{ding8191@qq.com}

\author{Michael E. Zieve}
\address{
  Department of Mathematics,
  University of Michigan,
  530 Church Street,
  Ann Arbor, MI 48109-1043 USA
}
\email{zieve@umich.edu}
\urladdr{http://www.math.lsa.umich.edu/$\sim$zieve/}

\date{\today}

\begin{abstract}
We determine all permutation polynomials over $\F_{q^2}$ of the form $X^r A(X^{q-1})$ where, for some $Q$ which 
is a power of the characteristic of $\F_q$, we have $r\equiv Q+1\pmod{q+1}$ and all terms of $A(X)$ have degrees 
in $\{0,1,Q,Q+1\}$. We then use this classification to resolve eight conjectures and open problems from the 
literature. Our proof makes a novel use of geometric techniques in a situation where they previously did not 
seem applicable, namely to understand the arithmetic of high-degree rational functions over small finite fields, 
despite the fact that in this situation the Weil bounds do not provide useful information.
\end{abstract}

\maketitle


\section{Introduction}

A polynomial $f(X)\in\F_q[X]$ is called a \emph{permutation polynomial} if the function $\alpha \mapsto f(\alpha)$ 
defines a bijection of $\F_q$. Permutation polynomials arise in various contexts in math and engineering. They are 
of particular interest when $f(X)$ has a simple algebraic form, in which case the interplay between the algebraic 
and combinatorial perspectives yields interesting results and challenges.

In the past several years, over $100$ papers have addressed permutation polynomials over $\F_{q^2}$ having the form 
$f(X) := X^r A(X^{q-1})$ where $r$ is a positive integer and $A(X)\in\F_{q^2}[X]$, following the initial paper \cite{ZR} 
which restated the permutation property of $f(X)$ in terms of whether an associated rational function $g(X)\in\F_{q^2}(X)$ 
permutes the set of $(q+1)$-th roots of unity, or equivalently whether an associated $h(X) \in \F_q(X)$ permutes 
$\bP^1(\F_q) := \F_q\cup\{\infty\}$. The main advantage of these restatements is that even the simplest choices 
for $g(X)$ or $h(X)$ correspond to interesting permutation polynomials $f(X)$ over $\F_{q^2}$. For instance, most 
of the permutation polynomials in the literature having the form $X^r A(X^{q-1})$ correspond to cases where $g(X)$ 
and $h(X)$ have degree at most $3$. Conversely, all possibilities for $g(X)$ or $h(X)$ having degree at most $4$ 
have been classified \cite{DZ1}; we will show elsewhere that this classification quickly implies all previous results 
classifying permutation polynomials of the form $X^r A(X^{q-1})$ where $q$ and the coefficients of $A(X)$ can vary but 
$r$ and the degrees of the terms of $A(X)$ are prescribed.

Recently several authors have produced permutation polynomials over $\F_{q^2}$ of the form $X^r A(X^{q-1})$ for which 
the corresponding permutation rational function $h(X)$ over $\F_q$ can have arbitrarily large degree. This led to 
a series of conjectures and open problems seeking classifications of all permutation polynomials for a series of 
choices of $r$ and the degrees of the terms of $A(X)$. In this paper we resolve all of these conjectures and open 
problems, by proving the following result:

\begin{thm}\label{fir}
Write $q:=p^k$ and $Q:=p^\ell$ where $p$ is prime and $k$ and $\ell$ are positive integers, and let $r$ be a positive 
integer such that $r\equiv Q+1\pmod{q+1}$. Write $A(X):=aX^{Q+1}+bX^Q+cX+d$ with $a,b,c,d\in\F_{q^2}$. Then the polynomial 
$X^r A(X^{q-1})=aX^{r+qQ+q-Q-1}+bX^{r+qQ-Q}+cX^{r+q-1}+dX^r$ permutes\/ $\F_{q^2}$ if and only if all of the following hold:
\renewcommand{\labelenumi}{\emph{(\arabic{enumi})}}
\begin{enumerate}
\item $\gcd(r,q-1)=1$;
\item $p=2$;
\item $e:=a^{q+1}+b^{q+1}+c^{q+1}+d^{q+1}$ is nonzero;
\item $(ab^q+cd^q)^Q = e^{Q-1} (ac^q+bd^q)$; and
\item writing $m$ for the largest integer of the form $2^i$ (with $i\ge 0$) which divides $\gcd(k,\ell)$, we have
\[
\Tr_{\F_q/\F_{2^m}}\Bigl(\frac{b^{q+1}+c^{q+1}}e\Bigr)=\frac{\lcm(k,\ell)}m.
\]
\end{enumerate}
\end{thm}

Since the conditions on $a,b,c,d$ in Theorem~\ref{fir} are complicated, we now state an alternate version of the result 
which is more useful in some situations (for instance, if one wishes to count the number of permutation polynomials of 
this form, or to produce explicit examples). Here $\mu_{q+1}$ denotes the set of $(q+1)$-th roots of unity in $\F_{q^2}^*$, 
and if $n$ is a positive integer then we write $\ord_2(n)$ for the largest nonnegative integer $i$ such that $2^i\mid n$.

\begin{thm}\label{fir2}
Write $q:=p^k$ and $Q:=p^\ell$ where $p$ is prime and $k$ and $\ell$ are positive integers, and let $r$ be a positive 
integer such that $r\equiv Q+1\pmod{q+1}$. Write $A(X):=aX^{Q+1}+bX^Q+cX+d$ with $a,b,c,d\in\F_{q^2}$. Then $X^r A(X^{q-1})$ 
permutes\/ $\F_{q^2}$ if and only if $p=2$, $\gcd(r,q-1)=1$, and $A(X)=\delta A_0(\gamma X)$ for some $\gamma\in\mu_{q+1}$, 
some $\delta\in\F_{q^2}^*$, and some $A_0(X)\in\F_{q^2}[X]$ such that one of the following holds:
\renewcommand{\labelenumi}{\emph{(\arabic{enumi})}}
\begin{enumerate}
\item $\ord_2(k)\le\ord_2(\ell)$ and there exist $\alpha,\beta\in\F_{q^2}\setminus\F_q$ for which
\begin{align*}
A_0(X)&=(\alpha^{Q+1}+\beta)X^{Q+1}+(\alpha^{q+Q}+\beta)X^Q\\&\qquad+(\alpha^{qQ+1}+\beta)X+(\alpha^{qQ+q}+\beta);
\end{align*}
\item $\ord_2(k)\ne\ord_2(\ell)$ and either $A_0(X)=X^{Q+1}$ or there exist $\alpha,\beta\in\F_{q^2}\setminus\mu_{q+1}$ for which
\begin{align*}
A_0(X)&=(\alpha^{qQ+q}+\beta)X^{Q+1}+(\alpha^{qQ}+\alpha\beta)X^Q\\&\qquad+(\alpha^q+\alpha^Q\beta)X+(1+\alpha^{Q+1}\beta);
\end{align*}
\item $\ord_2(k)\ge\ord_2(\ell)$ and either $A_0(X)=X^Q$ or there exist $\alpha,\beta\in\F_{q^2}\setminus\mu_{q+1}$ for which
\begin{align*}
A_0(X)&=(\alpha^{qQ}+\alpha^q\beta)X^{Q+1}+(\alpha^{qQ+1}+\beta)X^Q\\&\qquad+(1+\alpha^{q+Q}\beta)X+(\alpha+\alpha^Q\beta).
\end{align*}
\end{enumerate}
\end{thm}

We now explain how these results differ from all previous results in the subject. As noted above, the results from 
\cite{ZR} reduce the proofs of Theorem~\ref{fir} and Theorem~\ref{fir2} to determining when an associated rational function 
$g(X) \in \F_{q^2}(X)$ permutes $\mu_{q+1}$, or equivalently an associated rational function $h(X) \in \F_q(X)$ permutes 
$\bP^1(\F_q)$. A standard approach to investigating permutation rational functions, dating back to \cite{BSD}, \cite{DL}, 
and \cite{Hayes}, argues that if $h(X)\in\F_q(X)$ permutes $\bP^1(\F_q)$ then the (possibly reducible) curve $h(X)=h(Y)$ has 
no non-diagonal $\F_q$-rational points, which by Weil's bounds implies that if $q$ is sufficiently large compared to $\deg(h)$ 
then the diagonal is the only geometrically irreducible component of $h(X)=h(Y)$ which is defined over $\F_q$. This is the key 
first step in many important papers in the subject, since it enables one to use techniques from algebraic geometry, 
Galois theory, and group theory; e.g., cf.\ \cite{DZ1,DZexc,FGS,GMS,GRZ,GTZ,GZ}. However, this approach is only useful when 
$q$ is large compared to $\deg(h)$, since otherwise a geometrically irreducible component of $h(X)=h(Y)$ defined over $\F_q$ 
can have no $\F_q$-rational points. Since our main results include cases where $q<\deg(h)$, we are forced to introduce 
a completely new approach which does not rely on Weil's bound.

As noted above, our proofs of Theorem~\ref{fir} and Theorem~\ref{fir2} begin by using simple results from \cite{Zlem} and 
\cite{ZR} to reduce to the problem of determining when an associated rational function $g(X)$ permutes $\mu_{q+1}$. Thus 
the bulk of the proofs of the above results consists of the proof of the following result, which is of independent interest.

\begin{thm}\label{main}
Write $q:=p^k$ and $Q:=p^\ell$ where $p$ is prime and $k$ and $\ell$ are positive integers. Pick $a,b,c,d\in\F_{q^2}$ which are 
not all zero, and write $A(X):=aX^{Q+1}+bX^Q+cX+d$ and $B(X):=d^q X^{Q+1}+c^q X^Q + b^q X + a^q$. Then the following are equivalent:
\renewcommand{\labelenumi}{\emph{(\arabic{enumi})}}
\begin{enumerate}
\item $A(X)$ has no roots in $\mu_{q+1}$ and $g(X):=B(X)/A(X)$ permutes $\mu_{q+1}$;
\item conditions \emph{(2)--(5)} of Theorem~\emph{\ref{fir}} all hold;
\item $p=2$ and $A(X)=\delta A_0(\gamma X)$ for some $\gamma\in\mu_{q+1}$, some $\delta\in\F_{q^2}^*$, and some $A_0(X)\in\F_{q^2}[X]$ 
such that one of the conditions \emph{(1)--(3)} of Theorem~\emph{\ref{fir2}} holds.
\end{enumerate}
\end{thm}

Our proof of Theorem~\ref{main} proceeds in three steps. Loosely speaking, in step 1 we determine information about the geometry 
of $g(X)$ as a function $\bP^1\to\bP^1$; in step 2 we show that certain possibilities for this geometry prevent $g(X)$ from permuting 
$\mu_{q+1}$; and in step 3 we determine all permutations coming from the remaining geometric possibilities. More formally, the steps 
are as follows:
\begin{enumerate}
\item Show that if $A(X)$ has no roots in $\mu_{q+1}$ and $g(X)$ is nonconstant then either two points of $\bP^1(\mybar\F_q)$ 
each have a unique $g$-preimage or one point of $\bP^1(\mybar\F_q)$ has $g$-preimages with ramification indices $1$ and~$Q$.
\item Show that if some point of $\bP^1(\mybar\F_q)$ has $g$-preimages with ramification indices $1$ and $Q$ then $g(X)$ does 
not permute $\mu_{q+1}$ (in fact, we prove a more general result, cf.\ Theorem~\ref{lpp}).
\item Determine all choices of $A(X)$ which have no roots in $\mu_{q+1}$ in case $g(X)$ permutes $\mu_{q+1}$ and two points of 
$\bP^1(\mybar\F_q)$ each have a unique $g$-preimage.
\end{enumerate}
Each of these three steps presents new types of challenges. In the first step we go a long way towards determining the ramification 
of all members of each of infinitely many four-parameter families of coverings; \emph{a priori} it is not clear that this problem is 
feasible. The second step is entirely new, in that it uses purely geometric information in order to determine when $g(X)$ permutes 
$\mu_{q+1}$. We note that this type of conclusion can be proved when $q>4(Q+1)^4$ by using Galois theory and Weil's bound, but we prove 
the result without any assumption on the relative sizes of $q$ and $Q$, so that our proof cannot use Weil's bound and hence requires 
an entirely new approach. Finally, the third step combines geometric arguments with a series of elementary (but tricky) computations.

The hard part in our work is showing that $g(X)$ is not a permutation when item (2) or (3) of Theorem~\ref{main} does not hold. 
The converse implication is much easier, as it just involves computing the denominator of $\rho\circ X^{Q+1}\circ\sigma$ for certain 
degree-one $\rho,\sigma\in\F_{q^2}(X)$, and likewise computing the product of the denominator of $\rho\circ X^{Q-1}\circ\sigma$ with 
the product of the numerator and denominator of $\sigma(X)$ for certain degree-one $\rho,\sigma\in\F_{q^2}(X)$. This approach allows 
for a short proof that item (3) in Theorem~\ref{main} implies item (1) (cf.\ \cite{Zx}). One can also use it to show that (2) implies 
(1). But completely different ideas are needed in order to show that (1) implies (2) or (3). All previous results showing that (1) 
implies (2) or (3) in some special case only applied when $Q=2$, and are immediate consequences of the classification of degree-$3$ 
permutation rational functions (which is proved in a few lines in \cite[Thm.~1.3]{DZ1}, cf.\ Lemma~\ref{deg3perm}).

We note that the permutation condition in Theorems~\ref{fir} and \ref{fir2} immediately implies that $p=2$, since $\alpha$ and $-\alpha$ 
have the same image under $X^r A(X^{q-1})$. Thus, we could have assumed that $p=2$ in those results without significant loss. However, 
it is quite difficult to show that item (1) in Theorem~\ref{main} implies that $p=2$. Our perspective is that the most fundamental 
objects in our study are the functions $g(X)$ on $\mu_{q+1}$, so that it is natural to examine when such functions $g(X)$ permute 
$\mu_{q+1}$ in case $p$ is odd, even though that situation cannot yield permutations of $\F_{q^2}$.  This perspective is supported by the fact that the statements of Theorems~\ref{fir} and \ref{fir2} are complicated, but all the bijective functions $g(X)$ in Theorem~\ref{main} turn out to have the simple form $\rho\circ X^n\circ\sigma$ for some $n\in\{Q-1,Q+1\}$ and some degree-one $\rho,\sigma\in\F_{q^2}(X)$.

We will use the above results to resolve eight conjectures and open problems from the literature. Six of these describe all 
permutation polynomials among certain classes of polynomials: the conjecture in \cite[p.~4854]{LXZ}, the two open problems 
in \cite[Open problem]{ZKP} and \cite[Rem.~2]{ZKP}, the open problem \cite[Open problem 2]{ZKPT}, and the two conjectures in 
\cite[pp.~5 and 20]{ZLKPT}. The final two conjectures are \cite[Conj.~19]{LLHQ} and \cite[Rem.~2]{LHXZ}, which describe all 
permutations among a certain class of functions from $\F_q\times\F_q$ to itself that are defined by a pair of bivariate polynomials.  
Intriguingly, these last two conjectures arose in the context of boomerang attacks against butterfly structures in cryptography.  
In particular, our Corollary~\ref{pre5b}
determines all instances of the generalized closed butterfly map introduced in \cite{LTYW} which permute $\F_q\times\F_q$.  In light of \cite[Thm.~1]{LLHQ} and  \cite[Thm.~2(2)]{LHXZ2}, it follows that if such a map is a permutation then it has boomerang uniformity $4$, 
and also it is linearly equivalent to a Gold function over $\F_{q^2}$. We refer the interested reader to \cite{LLHQ,LHXZ2,LTYW} 
for the relevant definitions.

We note that some of the above conjectures follow easily from our results, while others require significant additional work. 
In particular, our proof of the conjecture from \cite[Rem.~2]{LHXZ} relies on a new polynomial identity of independent interest 
(Theorem~\ref{identity}). Unexpectedly, it turns out that this identity provides a new proof of a result of Cusick and M\"uller 
about image sizes of certain polynomials in $\F_q[X]$ which are \emph{not} permutations.

Our classification result immediately yields several further classification results of permutation polynomials $f(X)$ 
over $\F_{q^2}$, by composing $f(X)$ with permutation monomials and reducing mod $X^{q^2}-X$. Explicitly, we make the 
following definition:

\begin{defn}\label{me}
We say that $f,g\in\F_q[X]$ are \emph{multiplicatively equivalent} if $g(X)\equiv \beta f(\alpha X^n)\pmod{X^q-X}$ for some 
$\alpha, \beta \in \F_q^*$ and some positive integer $n$ such that $\gcd(n,q-1)=1$.
\end{defn}

Plainly this is an equivalence relation on $\F_q[X]$, and if $f,g\in\F_q[X]$ are multiplicatively equivalent then $f(X)$ 
permutes $\F_q$ if and only if $g(X)$ permutes $\F_q$. Moreover, if $f,g\in\F_q[X]$ are multiplicatively equivalent and 
$\deg(g)<q$ then $g(X)$ has at most as many terms as does $f(X)$.

\begin{rmk}
The above notion has been called ``quasi-multiplicative equivalence'' in some previous papers, and the term ``multiplicative equivalence'' 
has been used for each of two different notions. However, we cannot envision any situation in which either of the previous definitions 
of multiplicative equivalence would be preferable to the definition above, so we encourage subsequent authors to use the above definition.
\end{rmk}

In addition to the eight open problems mentioned above, our main results subsume $58$ previous results (in $32$ papers by $46$ authors),
%
%
once one replaces the polynomials in our main results by suitable multiplicatively equivalent polynomials; 
cf.~ Tables~\ref{tab2} and \ref{tab1}. We note that some of these previous results resolved three earlier 
conjectures from the literature.

\begin{table}[!htbp]
\caption{Previous results subsumed by Theorems~\ref{fir}--\ref{main}, I}
\label{tab2}
\renewcommand{\arraystretch}{1.1}
\begin{tabular}{|c|c|}
\hline
$Q=4$&$Q=8$\\ \hline\hline

\cite[Thm.~3.4 and 3.5 and Conj.~2]{GS} &
%

\cite[Thm.~4.9]{XCP}

%

\\ \hline

\cite[Thm.~3.7 and 3.8]{LQCL} & 
%

\\ \hline

\cite[Thm.~2.4 and 2.8]{LQLF} & 
%

\\ \hline

\cite[Thm.~3 and 6]{LHsevtri} & 
%

\\ \hline

\cite[Thm.~4.2 and 4.4]{WYDM} & 
%

\\ \hline

\cite[Thm.~4.3 and 4.4]{ZHF} & 
%

\\ \hline

\end{tabular}
\end{table}

\begin{table}[!htbp]
\caption{Previous results subsumed by Theorems~\ref{fir}--\ref{main}, II}
\label{tab1}
\renewcommand{\arraystretch}{1.1}
\begin{tabular}{|c|c|}
\hline
$Q=2^\ell$, $\ell$ arbitrary&$Q=2$\\ \hline\hline

\cite[Thm.~4.2 and 4.3]{BQ} &
%
%

\cite[Thm.~3.6]{Bartoli}
%

\\ \hline
\cite[Thm.~3.1]{GWSZL} &
%
%

\cite[Prop.~3.2 and Thm.~3.4]{BQ}
%

\\ \hline
\cite[Thm.~1 and Prop.~15]{LLHQ} & 
%
%
%

\cite[Thm.~3.2 and Cor.~3.7]{DQWYY}
%

\\ \hline
\cite[Thm.~1 and 2]{LHnewtri} & 
%
%

\cite[Thm.~5.9]{FHL2}
%

\\ \hline
\cite[Thm.~1 and 2]{LHXZ} & 
%
%

\cite[Thm.~2]{Houclasstri}
%

\\ \hline
\cite[Thm.~1]{LXZ} & 

\cite[Thm.~B]{Houdettri}
%

\\ \hline
\cite[Thm.~3.1]{WZZ} &
%

\cite[Thm.~B]{Houdettri2}
%

\\ \hline
\cite[Thm.~3.15 and Cor.~3.7--3.14]{WYDM} &
%

\cite[Thm.~1.1]{Houclasstri2}
%

\\ \hline 
\cite[Thm.~3.1 and 3.2]{ZKP} & 

\cite[Thm.~4.9]{LQC}
%

\\ \hline
\cite[Thm.~3.3]{ZKPT} &
%

\cite[Thm.~3.6]{LQCL}
%

\\ \hline
\cite[Thm.~1.1 and 1.3]{ZLKPT} &
%

\cite[Thm.~1.2]{LQLC}

\\ \hline
\cite[Thm.~1.1]{Zx}
%
%
& 

\cite[Thm.~1]{TLZ}

\\ \hline
& 

\cite[Thm.~1]{TZHnew}
%

\\ \hline
&

\cite[Thm.~1]{TZHclass}
%

\\ \hline
& 

\cite[Conj.~1]{TZHL}
%

\\ \hline
& 

\cite[Thm.~1 and Open problem]{TZLH}
%

\\ \hline
& 

\cite[Thm.~4.1 and 4.2]{ZHF}
%

\\ \hline

\end{tabular}
\end{table}


\FloatBarrier

This paper is organized as follows. In the next section we introduce notation and recall several background results 
we need for our proofs. In section~\ref{sec:geom} we determine geometric properties of the rational functions $g(X)$ 
under consideration. In section~\ref{sec:npp} we show that if $g(X)$ has certain geometric properties then $f(X)$ 
cannot permute $\F_{q^2}$. In section~\ref{sec:param} we combine the results of the previous two sections to prove 
Theorem~\ref{fir2}. In section~\ref{sec:Q} we translate the geometric properties from section~\ref{sec:geom} into 
conditions on the coefficients, and then in section~\ref{sec:main} we prove Theorem~\ref{fir} and Theorem~\ref{main}. 
In the final two sections we prove an identity of bivariate polynomials which implies the Cusick--M\"uller result 
about images of certain non-permutation polynomials \cite{CM}, and then use this identity and other ingredients to 
resolve eight conjectures and open problems from the literature.

\section{Preliminaries}

We use the following notation in this paper:

\begin{itemize}
\item $q$ is a prime power,
\item if $K$ is a field then $\mybar K$ is an algebraic closure of $K$,
\item if $n$ is a positive integer then $\mu_n$ is the set of $n$-th roots of unity in $\mybar\F_q$,
\item if $K$ is a field then $\bP^1(K):=K\cup\{\infty\}$ is the set of $K$-rational points on $\bP^1$,
\item if $s$ is a power of $2$ and $t$ is a power of $s$ then we define $\Tr_{\F_t/\F_s}(X)$ to be the polynomial 
$X+X^s+X^{s^2}+X^{s^3}+\dots+X^{t/s}$ ,
\item if $n$ is a positive integer then $\ord_2(n)$ denotes the largest integer $i\ge 0$ for which $2^i\mid n$.
\end{itemize}


\subsection{Self-conjugate reciprocal polynomials}

We now recall some simple results about self-conjugate reciprocal polynomials, which are defined as follows.

\begin{notation}
For $g(X)\in\F_{q^2}(X)$ we define $g^{(q)}(X)$ to be the rational function obtained from $g(X)$ 
by raising every coefficient to the $q$-th power. For any nonzero $A(X)\in\F_{q^2}[X]$ we define 
$\widehat A(X):=X^{\deg(A)} A^{(q)}(1/X)$. Explicitly, if $A(X)=\sum_{i=0}^n \alpha_i X^i$ with 
$\alpha_i\in\F_{q^2}$ and $\alpha_n\ne 0$ then $A^{(q)}(X)=\sum_{i=0}^n \alpha_i^q X^i$ and 
$\widehat A(X)=\sum_{i=0}^n \alpha_i^q X^{n-i}$.
\end{notation}

\begin{defn}\label{SCR}
We say that a nonzero $A(X)\in\F_{q^2}[X]$ is \emph{self-conjugate reciprocal} (or SCR for short) 
if $\widehat A(X)=\alpha A(X)$ for some $\alpha\in\F_{q^2}$.
\end{defn}

The next lemma is immediate from the definitions.

\begin{lemma}\label{SCRbasics} 
All of the following hold:
\begin{itemize}
\item If $A(X)\in\F_{q^2}[X]$ is SCR then $\widehat A(X)/A(X)\in\mu_{q+1}$.
\item For nonconstant $g_1,g_2\in\F_{q^2}(X)$ we have $(g_1\circ g_2)^{(q)}=g_1^{(q)}\circ g_2^{(q)}$.
\item If $A(X)\in\F_{q^2}[X]$ is nonzero and $\alpha\in\mybar\F_q^*$ then the multiplicity of $\alpha$ as a root 
of $A(X)$ equals the multiplicity of $\alpha^{-q}$ as a root of $\widehat A(X)$.
\item $A(X)\in\F_{q^2}[X]$ is SCR if and only if the multiset of roots of $A(X)$ is preserved by the function 
$\alpha\mapsto\alpha^{-q}$. 
\item Every degree-$1$ SCR polynomial has a root in $\mu_{q+1}$.
\item If $\alpha\in\F_{q^2}^*$ and $\beta\in\F_q$ then $\alpha X^2+\beta X+\alpha ^q$ is SCR.
\end{itemize}
\end{lemma}

We will also use the following simple result.

\begin{lemma}\label{quadscr}
If $q$ is even and $A(X):=\alpha X^2+\beta X+\alpha^q$ with $\alpha\in\F_{q^2}$ and $\beta\in\F_q^*$, then 
the following are equivalent:
\renewcommand{\labelenumi}{\emph{(\arabic{enumi})}}
\begin{enumerate}
\item \label{quadscr1} $A(X)$ has at least one root in $\mu_{q+1}$;
\item \label{quadscr2} $A(X)$ has two distinct roots in $\mu_{q+1}$;
\item \label{quadscr3} $\Tr_{\F_q/\F_2}(\alpha^{q+1}/\beta^2)=1$.
\end{enumerate}
\end{lemma}

\begin{proof}
We may assume $\alpha\ne 0$, since otherwise the result is immediate. Each root $\gamma$ of $A(X)$ is a root of 
$(\alpha/\beta^2)A(X)=(\alpha/\beta)^2X^2+(\alpha/\beta)X+\alpha^{q+1}/\beta^2$, so that $(\alpha/\beta)\gamma$ 
is a root of $X^2+X+\alpha^{q+1}/\beta^2$. Writing $\epsilon:=\Tr_{\F_q/\F_2}\bigl(\frac{\alpha}\beta\gamma\bigr)$, 
it follows that
\[
\Bigl(\frac{\alpha}\beta\gamma\Bigr)^q+\frac{\alpha}\beta\gamma = \epsilon^2+\epsilon
= \Tr_{\F_q/\F_2}\Bigl(\frac{\alpha^2}{\beta^2}\gamma^2+\frac{\alpha}\beta\gamma\Bigr)
= \Tr_{\F_q/\F_2}\Bigl(\frac{\alpha^{q+1}}{\beta^2}\Bigr).
\]
Multiply by $\gamma$ to get
\begin{align*}
\Bigl(\frac{\alpha}\beta\Bigr)^q\gamma^{q+1} &= \frac{\alpha}\beta\gamma^2 + \gamma \Tr_{\F_q/\F_2}\Bigl(\frac{\alpha^{q+1}}{\beta^2}\Bigr) \\
&= \gamma+\frac{\alpha^q}\beta + \gamma \Tr_{\F_q/\F_2}\Bigl(\frac{\alpha^{q+1}}{\beta^2}\Bigr),
\end{align*}
so that $\gamma^{q+1}=1$ if and only if
\[
\Bigl(\frac{\alpha}\beta\Bigr)^q+\frac{\alpha^q}\beta = \gamma\biggl(1+\Tr_{\F_q/\F_2}\Bigl(\frac{\alpha^{q+1}}{\beta^2}\Bigr)\biggr).
\]
Since $\gamma\ne 0$ and the left side of this equation is zero, it follows that $\gamma^{q+1}=1$ if and only if
\eqref{quadscr3} holds. Here $\gamma$ is an arbitrary root of $A(X)$, so we have shown that \eqref{quadscr3} 
is equivalent to both \eqref{quadscr1} and \eqref{quadscr2}.
\end{proof}

\begin{rmk}
We will give a comprehensive treatment of SCR polynomials in a forthcoming paper.
\end{rmk}


\subsection{Rational functions}

We make the following conventions about rational functions. Let $K$ be a a field and let $g(X)=N(X)/D(X)$ where 
$N,D\in K[X]$ with $D(X)$ monic. Let $C(X)$ be the monic greatest common divisor of $N(X)$ and $D(X)$ in $K[X]$, 
and write $N(X)=C(X)N_0(X)$ and $D(X)=C(X)D_0(X)$ with $N_0,D_0\in K[X]$. We make no distinction between $g(X)$ 
and $g_0(X):=N_0(X)/D_0(X)$. Thus, we view $g(X)$ as defining a function $\bP^1(K)\to\bP^1(K)$ given by 
$\alpha\mapsto g_0(\alpha)$, so that in particular $g(X)$ is defined at elements $\alpha\in K$ even if 
$N(\alpha)=D(\alpha)=0$. We refer to $N_0(X)$ and $D_0(X)$ as the numerator and denominator of $g(X)$, 
respectively, and we define $\deg(g):=\max(\deg(N_0),\deg(D_0))$ if $g(X)\ne 0$. We say that a nonconstant 
$g(X)\in K(X)$ is \emph{separable} if the field extension $K(x)/K(g(x))$ is separable, where $x$ is transcendental 
over $K$; it is known that $g(X)$ is separable if and only if $g(X)\notin K(X^p)$ where $p$ is the characteristic 
of $K$ (e.g., cf.\ \cite[Lemma~2.2]{DZ1}).

\begin{defn}
We say that nonconstant $f,g\in\F_q(X)$ are \emph{linearly equivalent} if $g=\rho\circ f\circ\sigma$ for some 
degree-one $\rho,\sigma\in\F_q(X)$.
\end{defn}
\noindent
Note that if $f,g\in\F_q(X)$ are linearly equivalent then $f(X)$ permutes $\bP^1(\F_q)$ if and only if $g(X)$ 
permutes $\bP^1(\F_q)$.

For any field $K$ and any degree-one $\rho(X)\in K(X)$, if we write 
\[
\rho(X):=(\alpha X+\beta)/(\gamma X+\delta)
\]
with $\alpha,\beta,\gamma,\delta\in K$ then we define 
\[
\rho^{-1}(X):=(\delta X-\beta)/(-\gamma X+\alpha).
\] 
This definition does not change if we multiply all of $\alpha,\beta,\gamma,\delta$ by a common element of $K^*$, 
and we note that $\rho\circ\rho^{-1}=X=\rho^{-1}\circ\rho$. In particular, if $\rho\in\F_{q^2}(X)$ has degree one 
then $(\rho^{(q)})^{-1}=(\rho^{-1})^{(q)}$.

We recall two simple results about degree-one rational functions from \cite{ZR}:

\begin{lemma}\label{mutomu}
A degree-one $\rho(X)\in\F_{q^2}(X)$ permutes $\mu_{q+1}$ if and only if $\rho(X)=(\beta^q X+\alpha^q)/(\alpha X+\beta)$ 
for some $\alpha,\beta\in\F_{q^2}$ with $\alpha^{q+1}\ne \beta^{q+1}$.
\end{lemma}

\begin{lemma}\label{mutoF}
A degree-one $\rho(X)\in\F_{q^2}(X)$ maps $\mu_{q+1}$ to\/ $\bP^1(\F_q)$ if and only if 
$\rho(X)=(\delta X+\gamma \delta^q)/(X+\gamma)$ for some $\gamma\in\mu_{q+1}$ and $\delta\in\F_{q^2}\setminus\F_q$.
\end{lemma}


\subsection{Connection between permutations of $\F_{q^2}$, $\mu_{q+1}$, and $\bP^1(\F_q)$}

In this section we recall some known results relating permutations of different sets. We begin with a special case 
of a lemma from \cite{Zlem}.

\begin{lemma}\label{old}
Write $f(X):=X^r A(X^{q-1})$ where $r$ is a positive integer, $q$ is a prime power, and $A(X)\in\F_{q^2}[X]$. Then $f(X)$ 
permutes\/ $\F_{q^2}$ if and only if $\gcd(r,q-1)=1$ and $g_0(X):=X^r A(X)^{q-1}$ permutes $\mu_{q+1}$.
\end{lemma}

The next lemma is immediate, and was introduced in \cite{ZR}.

\begin{lemma}\label{rewrite}
Write $g_0(X):=X^r A(X)^{q-1}$ where $r$ is an integer, $q$ is a prime power, and $A(X)\in\F_{q^2}[X]$ is nonzero. Then 
$g_0(X)$ maps $\mu_{q+1}$ into $\mu_{q+1}\cup\{0\}$, and if $A(X)$ has no roots in $\mu_{q+1}$ then $g_0(X)$ induces the 
same function on $\mu_{q+1}$ as does $g(X):=X^s A^{(q)}(1/X)/A(X)$, for any integer $s$ with $r\equiv s\pmod{q+1}$. In 
particular, $g_0(X)$ permutes $\mu_{q+1}$ if and only if $A(X)$ has no roots in $\mu_{q+1}$ and $g(X)$ permutes $\mu_{q+1}$.
\end{lemma}

We now translate the condition that $g(X)$ permutes $\mu_{q+1}$ to the condition that an associated rational function 
$h(X)$ permutes $\bP^1(\F_q)$, as was done in \cite{ZR}.

\begin{lemma}\label{gtoh}
Let $g(X)\in\F_{q^2}(X)$ be a nonconstant rational function having the form $g(X)=X^s A^{(q)}(1/X)/A(X)$ where 
$s$ is an integer, $q$ is a prime power, and $A(X)\in\F_{q^2}[X]$. Let $h(X):=\rho\circ g\circ\sigma^{-1}$ where 
$\rho,\sigma\in\F_{q^2}(X)$ are degree-one rational functions which map $\mu_{q+1}$ to\/ $\bP^1(\F_q)$. Then $h(X)$ 
is in\/ $\F_q(X)$, and $h(X)$ permutes\/ $\bP^1(\F_q)$ if and only if $g(X)$ permutes $\mu_{q+1}$.
\end{lemma}

\begin{proof}
Lemma~\ref{mutoF} implies that $\rho(X)=(\delta X+\gamma\delta^q)/(X+\gamma)$ for some $\gamma\in\mu_{q+1}$ and 
$\delta\in\F_{q^2}\setminus\F_q$. Thus
\[
\rho^{(q)}(X)=\frac{\delta^qX+\gamma^q\delta}{X+\gamma^q}=\frac{\gamma\delta^q+\delta X^{-1}}{\gamma+X^{-1}}=\rho(X^{-1}),
\]
and likewise $\sigma^{(q)}(X)=\sigma(X^{-1})$.  
Since $g^{(q)}(X)=g(X^{-1})^{-1}$, it follows that
\begin{align*}
h^{(q)}&=\rho^{(q)}\circ g^{(q)}\circ (\sigma^{(q)})^{-1} \\
&=\bigl(\rho\circ X^{-1}\bigr)\circ \bigl(X^{-1}\circ g\circ X^{-1}\bigr) \circ \bigl(X^{-1}\circ \sigma^{-1}\bigr) \\
&=\rho\circ g\circ\sigma^{-1} \\
&= h(X),
\end{align*}
so that $h(X)\in\F_q(X)$. Since $h(X)$ maps $\bP^1(\F_q)$ into itself, plainly $h(X)$ permutes $\bP^1(\F_q)$ 
if and only if $g(X)=\rho^{-1}\circ h\circ\sigma$ permutes $\mu_{q+1}$.
\end{proof}


\subsection{Ramification}

We now introduce the notation and terminology we will use when discussing ramification.

As usual, for any nonconstant $g(X)\in\mybar\F_q(X)$ and any $\alpha\in\bP^1(\mybar\F_q)$, the \emph{ramification index} 
$e_g(\alpha)$ is the multiplicity of $\alpha$ as a $g$-preimage of $g(\alpha)$; explicitly, for any degree-one 
$\rho,\sigma\in\mybar\F_q(X)$ such that $\sigma(0)=\alpha$ and $\rho(g(\alpha))=0$, the positive integer $e_g(\alpha)$ 
is the degree of the lowest-degree term of the numerator of $\rho\circ g\circ\sigma$. For $\beta\in\bP^1(\mybar\F_q)$ 
we define the \emph{ramification multiset} of $g(X)$ over $\beta$ to be the multiset $E_g(\beta)$ consisting of the 
ramification indices $e_g(\alpha)$ with $\alpha\in g^{-1}(\beta)$. Thus $E_g(\beta)$ is a collection of positive 
integers whose sum is $\deg(g)$. We make the convention that if $g(X)$ is constant then $E_g(\beta)$ is the empty 
multiset. We say that $\alpha\in\bP^1(\mybar\F_q)$ is a \emph{ramification point} of $g(X)$ if $e_g(\alpha)>1$, 
and that $\beta\in\bP^1(\mybar\F_q)$ is a \emph{branch point} of $g(X)$ if $\beta=g(\alpha)$ for some ramification 
point $\alpha$ of $g(X)$. We will use the following consequence of the Riemann--Hurwitz genus formula for the map 
$g\colon\bP^1\to\bP^1$, or equivalently for the function field extension $\mybar\F_q(x)/\mybar\F_q(g(x))$ where 
$x$ is transcendental over $\mybar\F_q$; e.g., cf.\ \cite[Thm.~3.4.13 and Thm.~3.5.1]{St}:

\begin{lemma}\label{rh}
Let $g(X)\in \mybar\F_q(X)$ be a separable rational function of degree $n$. For any finite subset $\Gamma$ 
of\/ $\bP^1(\mybar\F_q)$ we have
\[
2n-2 \ge \sum_{\alpha\in \Gamma} \bigl(e_g(\alpha)-1\bigr).
\]
\end{lemma}


\subsection{Low-degree permutation rational functions}

In this section we determine the ramification in separable permutation rational functions of degrees $3$ and $4$, 
which will be used in the proof of Theorem~\ref{lpp}. 

The classification of degree-$3$ permutation rational functions is proved in one page in \cite[Thm.~1.3]{DZ1}:

\begin{lemma}\label{deg3perm}
A separable degree-three $h(X)\in\F_q(X)$ permutes\/ $\bP^1(\F_q)$ if and only if one of the following holds:
\begin{itemize}
\item $q\equiv 2\pmod{3}$ and $h(X)$ is linearly equivalent to $X^3$;
\item $q\equiv 1\pmod{3}$ and $h(X)=\rho\circ X^3\circ\sigma^{-1}$ for some degree-one $\rho,\sigma\in\F_{q^2}(X)$ 
which map $\mu_{q+1}$ to\/ $\bP^1(\F_q)$;
\item $q\equiv 0\pmod{3}$ and $h(X)$ is linearly equivalent to $X^3-\alpha X$ for some nonsquare $\alpha\in\F_q^*$.
\end{itemize}
\end{lemma}

The ramification multisets of the rational functions in Lemma~\ref{deg3perm} are well-known:

\begin{cor}\label{deg3ram}
Every degree-$3$ permutation rational function $h(X)\in\F_q(X)$ has ramification multiset $[3]$ over each of its branch points.
\end{cor}

The classification of degree-$4$ permutation rational functions is more difficult, cf.\ \cite[Thm.~1.4]{DZ1}:

\begin{lemma}\label{deg4perm}
A separable degree-four $h(X)\in\F_q(X)$ permutes\/ $\bP^1(\F_q)$ if and only if one of the following holds:
\renewcommand{\labelenumi}{\emph{(\arabic{enumi})}}
\begin{enumerate}
\item $q$ is odd and $h(X)$ is linearly equivalent to \[\frac{X^4-2\alpha X^2-8\beta X+\alpha^2}{ X^3+\alpha X+\beta}\] 
for some $\alpha,\beta\in\F_q$ such that $X^3+\alpha X+\beta$ is irreducible in\/ $\F_q[X]$;
\item $q$ is even and $h(X)$ is linearly equivalent to $X^4+\alpha X^2+\beta X$ for some $\alpha,\beta\in\F_q$ such that 
$X^3+\alpha X+\beta$ has no roots in\/ $\F_q$;
\item $q\le 8$ and $h(X)$ is linearly equivalent to a rational function in Table~\emph{\ref{tab3}}.
\end{enumerate}
\end{lemma}

\begin{table}[!htbp]
\caption{Sporadic degree-$4$ permutation rational functions over $\F_q$}
\label{tab3}
\renewcommand{\arraystretch}{2.5}
\begin{tabular}[t]{|c|c|c|}
\hline
$q$&$h(X)$&Conditions \\ \hline\hline

$8$&$\displaystyle{\frac{X^4+\alpha X^3+X}{X^2+X+1}}$&$\alpha^3+\alpha=1$ \\ \hline

$7$&$X^4+3X$&\\ \hline

\multirow{2}{*}{$5$}&$\displaystyle{\frac{X^4+X+1}{X^2+2}}$& \\
&$\displaystyle{\frac{X^4+X^3+1}{X^2+2}}$& \\ \hline

\multirow{3}{*}{$4$}&$\displaystyle{\frac{X^4+\omega X}{X^3+\omega^2}}$&\multirow{3}{*}{$\omega^2+\omega=1$}\\
&$\displaystyle{\frac{X^4+X^2+X}{X^3+\omega}}$&\\
&$\displaystyle{\frac{X^4+\omega X^2+X}{X^3+X+1}}$&\\
\hline
\end{tabular}
\qquad
\begin{tabular}[t]{|c|c|}
\hline
$q$&$h(X)$\\ \hline\hline
\multirow{3}{*}{$3$}&$X^4-X^2+X$\\
&$\displaystyle{\frac{X^4+X+1}{X^2+1}}$\\
&$\displaystyle{\frac{X^4+X^3+1}{X^2+1}}$\\ \hline

\multirow{2}{*}{$2$}&$X^4+X^3+X$\\
&$\displaystyle{\frac{X^4+X^3+X}{X^2+X+1}}$\\ \hline
\end{tabular}
\end{table}

\begin{cor}\label{deg4ram}
If $h(X)\in\F_q(X)$ is a separable degree-four permutation rational function and $\gamma\in\bP^1(\mybar\F_q)$ 
has exactly two $h$-preimages then $q$ is odd and $E_h(\gamma)=[2,2]$.
\end{cor}

\begin{proof}
It is clear that the hypotheses are never satisfied in case (2), and it is routine to verify that the hypotheses 
are never satisfied in case (3). In case (1) the conclusion was shown in the proof of \cite[Thm.~1.4]{DZ1}, and 
can be verified directly by showing that the numerator of $h^3+16\alpha h+64\beta$ is a square and then applying 
Riemann--Hurwitz (Lemma~\ref{rh}).
\end{proof}

\FloatBarrier


\section{Geometric properties of $g(X)$}
\label{sec:geom}

In this section we prove the following result providing properties of a certain class of functions 
$\bP^1(\mybar\F_q)\to\bP^1(\mybar\F_q)$ which are used in the proofs of our main results. We note that the key 
geometric conclusion in the result is \eqref{A}, and the purpose of \eqref{B} and \eqref{C} is to provide 
information we will use in later sections to restate the geometric conclusion in terms of the coefficients.

\begin{thm}\label{geom}
Let $q$ and $Q$ be powers of the same prime, and let $a,b,c,d$ be elements of\/ $\F_{q^2}$ which are not 
all zero. Write $g(X):=B(X)/A(X)$ where $A(X):=aX^{Q+1}+bX^Q+cX+d$ and $B(X):=d^q X^{Q+1}+c^q X^Q+b^q X+a^q$.
\renewcommand{\theenumi}{\Alph{enumi}}
\renewcommand{\labelenumi}{\emph{(\Alph{enumi})}}
\begin{enumerate}
\item \label{A} At least one of the following holds:
\renewcommand{\theenumii}{\arabic{enumii}}
\renewcommand{\labelenumii}{\emph{(\Alph{enumi}\arabic{enumii})}}
\begin{enumerate}
\item\label{A1} $g(X)$ has ramification multiset $[1,Q]$ over some point in\/ $\bP^1(\mybar\F_q)$;
\item\label{A2} $g(X)=\rho\circ X^n\circ\sigma$ for some degree-one $\rho,\sigma\in\mybar\F_q(X)$ and some 
$n\in\{Q-1,Q+1\}$;
\item\label{A3} $g(X)$ is constant;
\item\label{A4}
$A(X)$ has at least one root in $\mu_{q+1}$.
\end{enumerate}
\item\label{B} If $q$ is even then the following are equivalent:
\begin{enumerate}
\item\label{B1} condition \eqref{A2} holds but \eqref{A4} does not hold;
\item\label{B2} $e:=a^{q+1}+b^{q+1}+c^{q+1}+d^{q+1}$ is nonzero,
\[
(ab^q+cd^q)^Q=e^{Q-1}(ac^q+bd^q),
\]
and either $U(X)\nmid A(X)$ or $U(X)$ has no roots in $\mu_{q+1}$, where $U(X):=(ab^q+cd^q)X^2+eX+a^qb+c^qd.
$
\end{enumerate}
\item\label{C} Suppose $q$ is even and \eqref{B2} holds, and let $\Lambda$ be the union of the set of roots of 
$W(X):=(bc+ad)X^2+eX+(bc+ad)^q $ in\/ $\mybar\F_q$ and the set consisting of $(2-\deg(W))$ copies of $\infty$. 
Then
\begin{enumerate}
\item \label{C1} each element of $\Lambda$ has a unique $g$-preimage in\/ $\bP^1(\mybar\F_q)$;
\item \label{C2} each root of $U(X)$ in\/ $\mybar\F_q$ is the unique $g$-preimage of some element of $\Lambda$;
\item \label{C3} $\gcd(A(X),B(X))$ divides $U(X)$;
\item \label{C4} we have $\deg(g)=Q-1$ if and only if $A(X)$ is divisible by $U(X)$ and $\{b,c,d\}\ne\{0\}$.
\end{enumerate}
\end{enumerate}
\end{thm}
\renewcommand{\theenumi}{\arabic{enumi}}
\renewcommand{\labelenumi}{(\arabic{enumi})}

\begin{proof}
Define
\begin{align*}
W(X)&:=(bc-ad)X^2 + (a^{q+1}-b^{q+1}-c^{q+1}+d^{q+1})X + (bc-ad)^q, \\
U(X)&:=(cd^q-ab^q)X^2+(-a^{q+1}-b^{q+1}+c^{q+1}+d^{q+1})X+c^qd-a^qb, \\
V(X)&:=(bd^q-ac^q)^{1/Q}X^2+(-a^{q+1}+b^{q+1}-c^{q+1}+d^{q+1})^{1/Q}X, \\
&\qquad + (b^qd-a^qc)^{1/Q}.
\end{align*}
Note that each of $W(X)$, $U(X)$, and $V(X)$ is either a constant times $X$ or a degree-$2$ SCR polynomial 
in $\F_{q^2}[X]$.

For any polynomial of the form $P(X):=\alpha X^2+\beta X+\gamma$ with $\alpha,\beta,\gamma\in\F_{q^2}$, define
\[
\Delta(P):=\beta^2-4\alpha\gamma.
\]
Thus if $\deg(P)=2$ then $\Delta(P)$ is the discriminant of $P(X)$. It is easy to check that
\begin{align}
\label{u} U(X) &= (d^q X+c^q)A(X)-(aX+b)B(X), \\
\label{vQ} V(X)^Q &= A(X)B'(X)-A'(X)B(X), \\
\label{Rg} U(X)\cdot V(X)^Q &= W(g(X))\cdot A(X)^2, \\
\label{Delta} \Delta(W)&=\Delta(U)=\Delta(V)^Q,
\end{align}
%
%
and
\begin{equation}
\begin{aligned}
&\text{if $q$ is even then \eqref{B2} holds if and only if $\Delta(V)\ne 0$, $U/V\in\F_{q^2}^*$,} \\
&\text{and either $U(X)\nmid A(X)$ or $U(X)$ has no roots in $\mu_{q+1}$.}
\end{aligned}
\end{equation}

Let $C(X)$ be the monic greatest common divisor of $A(X)$ and $B(X)$ in $\F_{q^2}[X]$. Then \eqref{u} 
and \eqref{vQ} imply that $C(X)$ divides $U(X)$ and $V(X)^Q$. Since $B(X)=X^{Q+1} A^{(q)}(1/X)$, we have 
$B(X^q)=X^{q(Q+1)} A(1/X)^q$, so the nonzero roots of $B(X)$ are the $(-q)$-th powers of the nonzero roots 
of $A(X)$, and moreover the multiplicity of any $\alpha\in\mybar\F_q^*$ as a root of $A(X)$ equals the 
multiplicity of $\alpha^{-q}$ as a root of $B(X)$. Since $A,B\in\F_{q^2}[X]$, the multiset of roots of 
each of $A(X)$ and $B(X)$ is preserved by the $q^2$-th power map. Thus the multiset of nonzero roots 
of $C(X)$ is preserved by the $(-q)$-th power map, so that $C(X)=X^t C_0(X)$ where $t\ge 0$ and 
$C_0(X)\in\F_{q^2}[X]$ is an SCR polynomial. Moreover, the multiset of roots of $A(X)$ in $\mu_{q+1}$ 
equals the multiset of roots of $B(X)$ in $\mu_{q+1}$, so if we write $A(X)=A_0(X)C(X)$ and 
$B(X)=B_0(X)C(X)$ with $A_0,B_0\in\F_{q^2}[X]$ then $A_0(X)$ and $B_0(X)$ have no roots in $\mu_{q+1}$.  
Here $A_0(X)$ and $B_0(X)$ are coprime, and \eqref{Rg} says
\begin{equation}\label{Rg2}
U(X)\cdot V(X)^Q = W\Bigl(\frac{B_0(X)}{A_0(X)}\Bigr)\cdot A_0(X)^2\cdot C(X)^2,
\end{equation}
where we note that $W_0(X):=W\bigl(B_0(X)/A_0(X)\bigr)\cdot A_0(X)^2$ is a polynomial. 

First suppose $g(X)$ is a constant $\lambda$, so that \eqref{A3} holds but \eqref{A2} and \eqref{B1} do not.  
By considering coefficients we find that $d^q=a\lambda$, $c^q=b\lambda$, and $\lambda^{q+1}=1$. It follows that 
if $q$ is even then $e:=a^{q+1}+b^{q+1}+c^{q+1}+d^{q+1}$ is zero. Thus \eqref{B} holds because neither \eqref{B1} 
nor \eqref{B2} does, and \eqref{C} is vacuously true.

Next suppose that $\deg(g)>0$ and at least one of $U(X)$, $V(X)$, and $W(X)$ is zero. Then \eqref{Rg} implies 
that $W(X)=0$ and either $U(X)=0$ or $V(X)=0$. It is straightforward to verify that in each case $A(X)$ has 
a root in $\mu_{q+1}$, and if $q$ is even then $e=0$, so that \eqref{A}, \eqref{B}, and \eqref{C} hold.
%
%

Henceforth assume that $g(X)$ is nonconstant and $U(X)$, $V(X)$, and $W(X)$ are all nonzero. Then each of $U(X)$, 
$V(X)$, and $W(X)$ has degree in $\{1,2\}$, so since $C(X)\mid U(X)$ we have $\deg(C)\le 2$. Now suppose that $A(X)$ 
has a root $\alpha$ in $\mu_{q+1}$, so that also $C(\alpha)=0$. In this case we need only show that if $q$ is even 
then \eqref{B2} does not hold. Suppose otherwise, so that $U(X)/V(X)$ is constant and each of $U(X)$, $V(X)$, 
and $W(X)$ is squarefree. Since $C(X)$ divides $U(X)$, we see that $C(X)$ is squarefree and $U(\alpha)=0$. Since 
$\alpha\ne 0$, it follows that $U(X)$ is a degree-$2$ SCR polynomial, so that $U(X)$ has a second root $\beta\ne\alpha$.  
Our hypothesis that $U(X)/V(X)$ is constant implies that $\beta$ has multiplicity $Q+1$ as a root of $UV^Q$. Since 
$\deg(C)>0$ we have $\deg(g)<Q+1$, so that \eqref{Rg2} implies that $C(\beta)=0$. Thus $U(X)\mid C(X)$ so that 
$U(X)\mid A(X)$, contradicting \eqref{B2}.

Henceforth we assume that $A(X)$ (and hence $C(X)$) has no roots in $\mu_{q+1}$. Since $C(X)$ is either a constant 
times $X$ or an SCR polynomial of degree at most $2$, it follows that $C(X)$ cannot have a root of multiplicity $2$.
Suppose for now that $C(0)=0$. Since $C(X)$ divides $U(X)$ and $V(X)^Q$, we must have $U(0)=V(0)=0$, so that both 
$U(X)$ and $V(X)$ are constants times $X$, whence $0\ne\Delta(U)=\Delta(W)$. Likewise $A(0)=B(0)=0$, so that $a=d=0$.  
Since $U(X)\ne 0$ we have $b^{q+1}\ne c^{q+1}$. Thus $g(X)=\rho\circ X^{Q-1}$ where $\rho(X):=(c^q X+b^q)/(bX+c)$. By 
\eqref{Rg}, if $\deg(W)=2$ then the only $g$-preimages of the two roots of $W(X)$ are $0$ and $\infty$, and if 
$\deg(W)=1$ then the only $g$-preimages of $0$ and $\infty$ are $0$ and $\infty$. Thus \eqref{A}, \eqref{B} and 
\eqref{C} hold in this case.

Henceforth assume that $C(0)\ne 0$, so that $C(X)$ is an SCR polynomial. Suppose for now that $\Delta(W)=0$, so that 
also $\Delta(U)=\Delta(V)=0$. Then each of $W(X)$, $U(X)$, and $V(X)$ is a degree-$2$ SCR polynomial with a unique 
root, so this root must be in $\mu_{q+1}$, and hence cannot be a root of $C(X)$. Since $C(X)\mid U(X)$, it follows 
that $C(X)=1$. Write $\alpha,\beta,\gamma$ for the unique roots of $U(X)$, $V(X)$, and $W(X)$, respectively. If 
$\alpha\ne\beta$ then \eqref{Rg} yields $E_g(\gamma) = [1,Q]$, which implies \eqref{A}, \eqref{B} and \eqref{C}. 
In the remaining case $\alpha=\beta$ we will obtain the contradiction $A(\beta)=0$. If $q$ is even then since 
$W(X)$ is not squarefree we have $a^{q+1}+b^{q+1}+c^{q+1}+d^{q+1}=0$, and from $U(\beta)=0$ and $V(\beta)^Q=0$ we obtain
\[
\beta^2=\frac{a^qb+c^qd}{ab^q+cd^q} \quad\text{ and }\quad
\beta^{2Q}=\frac{b^qd+a^qc}{bd^q+ac^q};
\]
it follows that
\[
A(\beta)^2=d^2+c^2\beta^2+\beta^{2Q}(b^2+a^2\beta^2)=0,
\]
yielding the desired contradiction.
%
%
If $q$ is odd then from $U(\beta)=0$ and $V(\beta)^Q=0$ we obtain
\[
\beta = \frac{a^{q+1}+b^{q+1}-c^{q+1}-d^{q+1}}{2(cd^q-ab^q)} \quad\text{ and }\quad
\beta^Q = \frac{a^{q+1}-b^{q+1}+c^{q+1}-d^{q+1}}{2(bd^q-ac^q)};
\]
these imply that
\[
A(\beta) = d+c\beta+\beta^Q(b+a\beta),
\]
which equals $a\cdot\Delta(W)$ divided by the leading coefficient of $4U(X)V(X)^Q$, and hence is zero.
%
%

Henceforth assume that $\Delta(W)\ne 0$, so that also $\Delta(U)\ne 0$ and $\Delta(V)\ne 0$. Thus all roots of each 
of $W(X)$, $U(X)$, and $V(X)$ have multiplicity $1$. Since $C(X)\mid U(X)$, it follows that all roots of $C(X)$ have 
multiplicity $1$. Suppose for now that $C(X)\ne 1$. Since $\deg(C)\ne 1$ and $\deg(C)\le 2$, it follows that $\deg(C)=2$.  
Here $C(X)$ has two distinct roots, each of which is a root of both $U(X)$ and $V(X)$, so we conclude that $U/C$ and 
$V/C$ are constant. By \eqref{Rg2}, it follows that $U(X)^{Q-1}/W_0(X)$ is constant. Note that $\deg(g)\le Q-1$ since 
$\deg(C)=2$, so that the sum of the elements of each $g$-ramification multiset is at most $Q-1$. If $\deg(W)=2$ 
then it follows that each root of $W(X)$ has a unique $g$-preimage; if $\deg(W)=1$ then each of $0$ and $\infty$ 
has a unique $g$-preimage. In either case the two preimages are the roots of $U(X)$, and we conclude that 
$g(X)=\rho\circ X^{Q-1}\circ\sigma$ for some degree-one $\rho,\sigma\in\mybar\F_q(X)$. Thus \eqref{A2} holds 
but \eqref{A4} does not hold, and both \eqref{B} and \eqref{C} hold.

The remaining possibility is $C(X)=1$. In this case $C(0)\ne 0$, so that $\{a,d\}\ne\{0\}$ and thus $\deg(g)=Q+1$. 
Let $\Lambda$ be the set of roots of $W(X)$ if $\deg(W)=2$, and $\Lambda:=\{0,\infty\}$ if $\deg(W)=1$. Then the 
multiset of $g$-preimages of elements of $\Lambda$ (counted with multiplicities) is the union of the multiset of 
roots of $UV^Q$ and the multiset consisting of $s := 2Q+2-\deg(UV^Q)$ copies of $\infty$. Writing $\Sigma$ for 
the union of the $g$-ramification multisets of the elements of $\Lambda$, it follows that $\Sigma=E_{UV^Q}(0)$ if 
$s=0$ and $\Sigma=E_{UV^Q}(0)\cup [s]$ if $s>0$. If $U(X)$ has a root $\alpha$ which is not a root of $V(X)$ then 
$\Sigma$ is either $[1,1,Q,Q]$ or $[1,Q,Q+1]$, so in either case $g(X)$ has ramification multiset $[1,Q]$ over 
some element of $\Lambda$, which implies \eqref{A}, \eqref{B} and \eqref{C}. Finally, assume that every root of 
$U(X)$ is a root of $V(X)$, so that $U(X)/V(X)$ is constant. Then $\Sigma=[Q+1,Q+1]$, so that each element of 
$\Lambda$ has a unique $g$-preimage, and conversely each root of $U(X)$ is the unique $g$-preimage of an element 
of $\bP^1(\mybar\F_q)$. It follows that $g=\rho\circ X^{Q+1}\circ\sigma$ for some degree-one $\rho,\sigma\in\mybar\F_q(X)$. 
Now assume in addition that $q$ is even. Then \eqref{B1} and \eqref{B2} hold, and since $\deg(g)=Q+1$ it remains 
only to show that if $U(X)\mid A(X)$ then $\{b,c,d\}=\{0\}$. So suppose that $U(X)\mid A(X)$. Then $U(X)$ cannot 
be SCR, since if it were then we would also have $U(X)\mid B(X)$, which is impossible since $C(X)=1$. Thus $U(0)=0$, 
so since $U(X)\mid A(X)$ we have $d=A(0)=0$. Since $C(X)=1$ we must have $B(0)\ne 0$, so that $a\ne 0$. Since $U(X)$ 
and $V(X)$ have no degree-$2$ terms, it follows that $b=c=0$, as desired. Thus \eqref{A}, \eqref{B} and \eqref{C} 
hold in this case, which completes the proof.
\end{proof}


\section{Near-polynomial permutations}
\label{sec:npp}

In this section we study permutation rational functions $h(X)$ over $\F_q$ which are ``nearly'' polynomials, in the 
sense that some $\gamma\in\bP^1(\mybar\F_q)$ has exactly two $h$-preimages. We determine all such permutation rational 
functions when the ramification indices of the points in $h^{-1}(\gamma)$ satisfy certain mild constraints. The main 
result is as follows.

\begin{thm}\label{lpp}
Let $h(X)\in\F_q(X)$ be a nonconstant rational function whose ramification multiset over some point 
$\gamma\in\bP^1(\mybar\F_q)$ is $[s,t]$, where $s$ and $t$ are positive integers and $\gcd(t,q+1)=1$. Then $h(X)$ 
permutes\/ $\bP^1(\F_q)$ if and only if $s \equiv 0 \pmod{q+1}$ and $\gamma \in \F_{q^2}\setminus\F_q$.
\end{thm}

We first prove the following result about permutations of $\mu_{q+1}$.

\begin{lemma}\label{coprime}
Write $g(X) := X^s (\alpha^qX^t-1) / (X^t-\alpha)$ where $s$ is any integer, $t$ is a positive integer coprime to $q+1$, 
and $\alpha\in\F_{q^2}^*\setminus\mu_{q+1}$. Then $g(X)$ permutes $\mu_{q+1}$ if and only if $s \equiv 0 \pmod{q+1}$.
\end{lemma}

\begin{proof}
If $s \equiv 0 \pmod{q+1}$ then $g(X)$ induces the same map on $\mu_{q+1}$ as $\rho(X) \circ X^t$, where
$\rho(X) := (\alpha^qX-1) / (X-\alpha)$. Thus $g(X)$ permutes $\mu_{q+1}$ since both $X^t$ and $\rho(X)$ permute $\mu_{q+1}$.

Conversely, suppose that $g(X)$ permutes $\mu_{q+1}$, so that 
\[
\sum_{\beta\in\mu_{q+1}} g(\beta)=\sum_{\beta\in\mu_{q+1}} \beta.
\]
The right side is fixed by multiplication by any nontrivial $(q+1)$-th root of unity, and hence equals $0$.
For any $\beta\in \mu_{q+1}$ we have $\beta^t-\alpha\ne 0$ and
\[
g(\beta) = \beta^s\cdot \frac{\alpha^q\beta^t-1}{\beta^t-\alpha} 
= -\beta^{s+t}\cdot \frac{\beta^{tq}-\alpha^q}{\beta^t-\alpha} = -\sum_{i=0}^{q-1} \alpha^{q-1-i} \beta^{s+t(i+1)},
\]
so that
\begin{equation}\label{hermite}
0 = \sum_{\beta\in\mu_{q+1}} g(\beta) = -\sum_{i=0}^{q-1}\alpha^{q-1-i}\sum_{\beta\in\mu_{q+1}} \beta^{s+t(i+1)}.
\end{equation}
For any $j\in\Z$, the summation $\sum_{\beta\in\mu_{q+1}} \beta^j$ is nonzero if and only if $j \equiv 0 \pmod{q+1}$. 
In particular, for any $i\in\Z$, $\epsilon_i:=\sum_{\beta\in\mu_{q+1}}\beta^{s+t(i+1)} \ne 0$ if and only if 
$s\equiv -t(i+1)\pmod{q+1}$. Since $\gcd(t,q+1)=1$, there is exactly one integer $i$ with $0\le i\le q$ 
for which $\epsilon_i\ne 0$. By \eqref{hermite} this distinguished integer $i$ must be $q$, so that 
$s\equiv -t(q+1)\equiv 0\pmod{q+1}$.
\end{proof}

The following consequence of Lemma~\ref{coprime} is not used in this paper, but is stated for its inherent interest.

\begin{cor}\label{corbin}
Let $r$ and $t$ be positive integers with $\gcd(t,q+1)=1$. Pick $\alpha \in \F_{q^2}^*$ and write 
$f(X) := X^r (X^{t(q-1)}-\alpha)$. Then $f(X)$ permutes\/ $\F_{q^2}$ if and only if $\gcd(r,q-1)=1$, 
$r \equiv t \pmod{q+1}$, and $\alpha \notin \mu_{q+1}$.
\end{cor}

\begin{proof}
This follows immediately from Lemma~\ref{coprime}, in light of Lemmas~\ref{old} and \ref{rewrite}.
\end{proof}

\begin{rmk}
The special case $r=1$ of Corollary~\ref{corbin} is \cite[Thm.~2]{Houbin}; the special case that $t=1$ and 
$1\le r\le q+1$ is \cite[Thm.~3.1]{LQC}. The proofs in \cite{Houbin} and \cite{LQC} involve complicated computations.
\end{rmk}

We now prove Theorem~\ref{lpp}.

\begin{proof}[Proof of Theorem~\ref{lpp}]
Let $p$ be the characteristic of $\F_q$, and let $\ell$ be the largest nonnegative integer for which $h(X)\in\F_q(X^{p^\ell})$.  
Then $h=h_0\circ X^{p^\ell}$ where $h_0(X)\in\F_q(X)$ is separable. Here $h(X)$ permutes $\bP^1(\F_q)$ if and only if $h_0(X)$ 
does, and $E_{h_0}(\gamma)=[s_0,t_0]$ where $s_0:=s/p^\ell$ and $t_0:=t/p^\ell$, so Theorem~\ref{lpp} holds for $h(X)$ if and 
only if it holds for $h_0(X)$. Thus we may replace $h(X)$ by $h_0(X)$ in order to assume that $h(X)$ is separable.

Write $n:=\deg(h)$, so that $n = s+t$. Since there are no separable degree-$2$ permutation rational functions (e.g., by 
\cite[Lemma~1.2]{DZ1}), we may assume that $n\ge 3$. Let $\alpha$ and $\beta$ be the $h$-preimages of $\gamma$, where 
$\alpha$ and $\beta$ have ramification indices under $h(X)$ being $s$ and $t$, respectively. Every $\lambda\in\bP^1(\mybar\F_q)$ 
satisfies $e_h(\lambda)=e_{h^{(q)}}(\lambda^q)$, so since $h(X)\in\F_q(X)$ we have $e_h(\lambda)=e_h(\lambda^q)$. Moreover, 
$h(\lambda^q)=h(\lambda)^q$, so the $q$-th power map preserves the set of branch points of $h(X)$ which have any prescribed 
ramification multiset, and also the $q$-th power map permutes the set of $h$-preimages of any element of $\bP^1(\F_q)$.
%
%

First suppose $\gamma\in\bP^1(\F_q)$. Then the $q$-th power map preserves $\{\alpha,\beta\}$, so that this set contains 
either zero or two elements of $\bP^1(\F_q)$. In particular, $h^{-1}(\gamma)\cap\bP^1(\F_q)$ cannot have size $1$, so $h(X)$ 
does not permute $\bP^1(\F_q)$. 

Next suppose $\gamma \notin \bP^1(\F_{q^2})$. Then $\gamma$, $\gamma^q$, and $\gamma^{q^2}$ are pairwise distinct, 
and they each have $h$-ramification multiset $[s,t]$. By the Riemann--Hurwitz formula (Lemma~\ref{rh}), it follows 
that $2n-2\ge 3\bigl(s+t-2\bigr)=3(n-2)$, so that $n\le 4$, whence $n\in\{3,4\}$. If $n=3$ then $\{s,t\}=\{1,2\}$, 
so Corollary~\ref{deg3ram} implies that $h(X)$ does not permute $\bP^1(\F_q)$. Thus we must have $n=4$. Since 
$\gcd(t,q+1)=1$ by hypothesis, in particular we cannot have $s=t=2$ in case $q$ is odd, so that Corollary~\ref{deg4ram} 
implies that $h(X)$ does not permute $\bP^1(\F_q)$.

The remaining possibility is that $\gamma \in \F_{q^2}\setminus\F_q$. Then $\gamma^q\ne\gamma$ and we have 
$h^{-1}(\gamma^q)=\{\alpha^q,\beta^q\}$ where $e_h(\alpha^q)=s$ and $e_h(\beta^q)=t$, so that in particular $\alpha^q\ne\beta$. 
Likewise, $h^{-1}(\gamma)=h^{-1}(\gamma^{q^2})=\{\alpha^{q^2},\beta^{q^2}\}$ where $e_h(\alpha^{q^2})=s$ and $e_h(\beta^{q^2})=t$, 
so that either $\alpha,\beta\in\F_{q^2}\setminus\F_q$ or $\beta=\alpha^{q^2}\in\F_{q^4}\setminus\F_{q^2}$, where in the latter 
case $s=t$.

First suppose $\alpha,\beta\in \F_{q^2}\setminus\F_q$. Then the degree-one rational functions $\rho(X):=(X-\gamma^q)/(X-\gamma)$ 
and $\sigma(X):=(X-\alpha^q)/(X-\alpha)$ map $\bP^1(\F_q)$ bijectively onto $\mu_{q+1}$, so that $g:=\rho\circ h\circ\sigma^{-1}$ 
permutes $\mu_{q+1}$ if and only if $h(X)$ permutes $\bP^1(\F_q)$. Here the poles of $g(X)$ are $\infty$ and 
$\delta:=\sigma(\beta) \in\F_{q^2}^*\setminus\mu_{q+1}$, with ramification indices $s$ and $t$, respectively, 
and likewise the zeroes of $g(X)$ are $0$ and $\sigma(\beta^q)=1/\sigma^{(q)}(\beta^q)=1/\delta^q$. It follows that 
$g(X)=\epsilon X^s (\delta^qX-1)^t/(X-\delta)^t$ for some $\epsilon\in\mybar\F_q^*$. Since $g=\rho\circ h\circ\sigma^{-1}$ 
maps $\mu_{q+1}$ into $\mu_{q+1}$, and also $X^s (\delta^q X-1)^t/(X-\delta)^t$ maps $\mu_{q+1}$ into $\mu_{q+1}$, 
we must have $\epsilon\in\mu_{q+1}$. Since $X^t$ permutes $\mu_{q+1}$ and $X^t\circ g_1=g\circ X^t$ where 
$g_1(X):=\epsilon_1X^s (\delta^q X^t-1)/(X^t-\delta)$ and $\epsilon=\epsilon_1^t$ with $\epsilon_1\in \mu_{q+1}$, 
we see that $h(X)$ permutes $\bP^1(\F_q)$ if and only if $g_1(X)$ permutes $\mu_{q+1}$. Finally, Lemma~\ref{coprime} 
says that $g_1(X)$ permutes $\mu_{q+1}$ if and only if $s\equiv 0\pmod{q+1}$, which concludes the proof when 
$\alpha,\beta\in \F_{q^2}\setminus\F_q$.

Finally, suppose that $\beta=\alpha^{q^2}\in \F_{q^4}\setminus\F_{q^2}$ and $s=t$. As above, $\rho(X):=(X-\gamma^q)/(X-\gamma)$ 
maps $\bP^1(\F_q)$ bijectively onto $\mu_{q+1}$, so that $g:=\rho\circ h\circ\rho^{-1}$ permutes $\mu_{q+1}$ if and only if 
$h(X)$ permutes $\bP^1(\F_q)$. Here the poles of $g(X)$ are $\delta:=\rho(\alpha)$ and $\rho(\beta)=\rho(\alpha^{q^2})=\delta^{q^2}$, 
and likewise the zeroes are $1/\delta^q$ and $1/\delta^{q^3}$, where each zero and pole has ramification index $t$. Writing 
$A(X):=(X-\delta)(X-\delta^{q^2})$ and $g_2(X):=X^2 A^{(q)}(1/X)/A(X)$, it follows that $g(X)=\epsilon g_2(X)^t$ for some 
$\epsilon\in\mybar\F_q^*$. Since $\alpha\in\F_{q^4}\setminus\F_{q^2}$, we have $\delta\in\F_{q^4}\setminus\F_{q^2}$, so that 
$A(X)\in\F_{q^2}[X]$ and $A(X)$ has no roots in $\F_{q^2}$. By Lemma~\ref{rewrite} we have $g_2(\mu_{q+1})\subseteq\mu_{q+1}$; 
since also $g(\mu_{q+1})\subseteq\mu_{q+1}$, it follows that $\epsilon\in\mu_{q+1}$. Thus $h(X)$ permutes $\bP^1(\F_q)$ 
if and only if $g_2(X)$ permutes $\mu_{q+1}$, or equivalently $h_2(X):=\rho^{-1}\circ g_2\circ\rho$ permutes $\bP^1(\F_q)$. 
By Lemma~\ref{gtoh} we have $h_2(X)\in\F_q(X)$, so that $h_2(X)$ cannot permute $\bP^1(\F_q)$ since there do not exist 
separable degree-$2$ permutation rational functions (e.g., by \cite[Lemma~1.2]{DZ1}).
\end{proof}

We conclude this section with the following reformulation of Theorem~\ref{lpp} in term of permutations of $\mu_{q+1}$.

\begin{cor}\label{lpp2}
Assume $A(X)\in\F_{q^2}[X]$ has no roots in $\mu_{q+1}$, and $r,s,t$ are integers with $\gcd(t,q+1)=1$. Suppose  
$g(X):=X^r A^{(q)}(1/X)/A(X)$ has ramification multiset $[s,t]$ over some $\gamma\in\bP^1(\mathbb F_q)$. Then $g(X)$ 
permutes $\mu_{q+1}$ if and only if $s\equiv 0\pmod{q+1}$ and $\gamma\in\bP^1(\F_{q^2})\setminus\mu_{q+1}$.
\end{cor}

\begin{proof}
This follows immediately from Lemma~\ref{gtoh} and Theorem~\ref{lpp}.
\end{proof}


\section{Geometrically cyclic bijections}
\label{sec:param}

In this section we prove Theorem~\ref{fir2} and show that items (1) and (3) in Theorem~\ref{main} are equivalent to 
one another. In light of Theorem~\ref{geom} and Corollary~\ref{lpp2}, we must consider rational functions which are 
compositions of $X^{Q+1}$ or $X^{Q-1}$ with certain degree-one rational functions. We first give a bijectivity criterion 
for such functions.

\begin{lemma}\label{cyclicperm}
Let $q$ be a power of a prime $p$, and assume that $g(X):=X^r A^{(q)}(1/X)/A(X) \in \F_{q^2}(X)$ has degree 
$n\ge 1$, where $r\in\Z$ and $A(X)\in\F_{q^2}[X]$ has no roots in $\mu_{q+1}$. Suppose there exist distinct 
$\beta_1,\beta_2\in\bP^1(\mybar\F_q)$ such that $\beta_i$ has a unique $g$-preimage $\alpha_i\in\bP^1(\mybar\F_q)$ 
for each $i$. Then the following are equivalent:
\renewcommand{\labelenumi}{\emph{(\arabic{enumi})}}
\begin{enumerate}
\item $g(X)$ permutes $\mu_{q+1}$.
\item At least one of the following holds:
\begin{itemize}
\item $\gcd(n,q-1) = 1$ and at least one $\alpha_i$ is in $\mu_{q+1}$;
\item $\gcd(n,q+1) = 1$ and at least one $\alpha_i$ is not in $\mu_{q+1}$.
\end{itemize}
\item At least one of the following holds:
\begin{itemize}
\item $\gcd(n,q-1) = 1$ and at least one $\beta_i$ is in $\mu_{q+1}$;
\item $\gcd(n,q+1) = 1$ and at least one $\beta_i$ is not in $\mu_{q+1}$.
\end{itemize}
\item At least one of the following holds:
\begin{itemize}
\item $\gcd(n,q-1)=1$ and $g(X)=\rho^{-1} \circ X^n \circ \sigma$ for some degree-one $\rho,\sigma\in \F_{q^2}(X)$ 
which map $\mu_{q+1}$ to\/ $\bP^1(\F_q)$;
\item $\gcd(n,q+1)=1$ and $g(X)=\rho^{-1} \circ X^n \circ \sigma$ for some degree-one $\rho,\sigma\in \F_{q^2}(X)$ 
which permute $\mu_{q+1}$.
\end{itemize}
\end{enumerate}
Moreover, if $g(X)=g_1(X^{p^\ell})$ where $\ell\ge 0$ and $g_1(X)\in\F_{q^2}(X)\setminus\F_{q^2}(X^p)$ has degree 
at least $2$ then in \emph{(4)} we may require in addition that $\sigma(\alpha_1)=\infty=\rho(\beta_1)$ and 
$\sigma(\alpha_2)=0=\rho(\beta_2)$.
\end{lemma}

\begin{proof}
We first reduce to the case that $g(X)$ is separable of degree at least $2$. Write $g(X) = g_0(X) \circ X^{p^\ell}$ 
where $g_0(X) \in \F_{q^2}(X)$ is separable of degree $n_0\ge 1$ and $n=n_0p^\ell$. Then $\alpha_i^{p^\ell}$ is the unique 
$g_0$-preimage of $\beta_i$ for $i\in\{1,2\}$, so that $g_0(X)$ satisfies the hypotheses of Lemma~\ref{cyclicperm}.  
We now show that $g(X)$ satisfies the conclusion of Lemma~\ref{cyclicperm} if and only if $g_0(X)$ does. Since 
$X^{p^\ell}$ induces a bijection on $\bP^1(\mybar\F_q)$ which restricts to a bijection of $\mu_{q+1}$, we see 
that $g(X)$ permutes $\mu_{q+1}$ if and only if $g_0(X)$ permutes $\mu_{q+1}$. Next, for $\epsilon\in\{1,-1\}$ 
we have $\gcd(n,q+\epsilon)=\gcd(n_0,q+\epsilon)$. Finally, for degree-one $\rho,\sigma\in\F_{q^2}(X)$, let 
$\widetilde{\sigma}(X)$ be the degree-one rational function obtained from $\sigma(X)$ by raising every coefficient 
to the $p^\ell$-th power, so that $X^{p^\ell}\circ\sigma=\widetilde{\sigma}\circ X^{p^\ell}$. Then $\sigma(X)$ permutes 
$\mu_{q+1}$ if and only if $\widetilde{\sigma}(X)$ does, and $\sigma(\mu_{q+1})=\bP^1(\F_q)$ if and only if 
$\widetilde{\sigma}(\mu_{q+1})=\bP^1(\F_q)$. Finally, $g(X)=\rho^{-1}\circ X^n\circ\sigma$ if and only if 
$g_0(X)=\rho^{-1}\circ X^{n_0}\circ\widetilde{\sigma}$. Thus each of conditions (1)--(4) holds for $g(X)$ if 
and only if the corresponding condition holds for $g_0(X)$. Hence in order to prove Lemma~\ref{cyclicperm} 
for $g(X)$, it suffices to prove the result for $g_0(X)$, so we may assume that $g(X)$ is separable. If $n=1$ 
then (2) and (3) are immediate, and (1) holds by Lemma~\ref{rewrite}, so that the second condition in (4) holds 
with $\rho(X)=X$ and $\sigma(X)=g(X)$. Thus the result holds when $n=1$, so we assume henceforth that $g(X)$ 
is separable of degree $n>1$. 

By Lemma~\ref{rewrite} we have $g(\mu_{q+1})\subseteq\mu_{q+1}$. The definition of $g(X)$ yields 
$g(X) = X^{-1} \circ g^{(q)}(X) \circ X^{-1}$, so for each $i\in\{1,2\}$ the unique $g$-preimage of 
$\beta_i^{-q}$ in $\bP^1(\mybar\F_q)$ is $\alpha_i^{-q}$. By Riemann--Hurwitz (Lemma~\ref{rh}), at most 
two elements of $\bP^1(\mybar\F_q)$ have a unique $g$-preimage, so that each $\alpha_i^{-q}$ is in 
$\{\alpha_1,\alpha_2\}$. Since $\alpha_1\ne\alpha_2$, we also have $\alpha_1^{-q}\ne\alpha_2^{-q}$. Thus 
if $\alpha_1^{-q}=\alpha_1$ then $\alpha_2^{-q}=\alpha_2$, so that $\alpha_1,\alpha_2\in\mu_{q+1}$, and thus 
also each $\beta_i=g(\alpha_i)$ is in $\mu_{q+1}$. The other possibility is that $\alpha_1^{-q}=\alpha_2$, 
so that $\alpha_2^{-q}=\alpha_1$, in which case $\alpha_1,\alpha_2\in\bP^1(\F_{q^2})\setminus\mu_{q+1}$, 
whence also $\beta_1,\beta_2\in\bP^1(\F_{q^2})\setminus\mu_{q+1}$. Moreover, as above each $\beta_i^{-q}$ 
is in $\{\beta_1,\beta_2\}$, so since $\beta_i\notin\mu_{q+1}$ we have $\beta_2=\beta_1^{-q}$.

First suppose $\alpha_1, \alpha_2, \beta_1, \beta_2$ are all in $\mu_{q+1}$. Pick $\gamma,\delta\in\F_{q^2}^*$ 
with $\gamma^{q-1}=\alpha_2/\alpha_1$ and $\delta^{q-1}=\beta_2/\beta_1$, so that $\gamma,\delta\notin\F_q$ 
since $\alpha_1\ne\alpha_2$ and $\beta_1\ne\beta_2$. Define $\sigma(X):=\gamma (X-\alpha_2)/(X-\alpha_1)$ 
and $\widetilde{\rho}(X):=\delta(X-\beta_2)/(X-\beta_1)$, so that $\sigma(X)$ and $\widetilde{\rho}(X)$ 
are degree-one rational functions in $\F_{q^2}(X)$. Then $\sigma(X)=(\gamma X-\alpha_1\gamma^q)/(X-\alpha_1)$, 
so that $\sigma(\mu_{q+1})=\bP^1(\F_q)$ by Lemma~\ref{mutoF}, and likewise $\widetilde{\rho}(\mu_{q+1})=\bP^1(\F_q)$.  
Thus $h(X) := \widetilde{\rho}(X) \circ g(X) \circ \sigma^{-1}(X)$ maps $\bP^1(\F_q)$ into $\bP^1(\F_q)$. Since 
$h^{-1}(\infty)=\{\infty\}$ and $h^{-1}(0)=\{0\}$, we have $h(X)=\epsilon X^n$ for some $\epsilon\in\mybar\F_q^*$, 
and since $h(\bP^1(\F_q))\subseteq\bP^1(\F_q)$ we must have $\epsilon\in\F_q^*$. Here $g(X)$ permutes 
$\mu_{q+1}$ if and only if $h(X)$ permutes $\bP^1(\F_q)$, or equivalently $\gcd(n,q-1)=1$. 
Moreover, $\rho(X):=\epsilon^{-1}\widetilde{\rho}(X)$ maps $\mu_{q+1}$ to $\bP^1(\F_q)$, 
and $g(X)=\rho^{-1}\circ X^n\circ \sigma$. Thus if each $\alpha_i$ and $\beta_i$ is in 
$\mu_{q+1}$ then (1) is equivalent to each of (2), (3), and (4).

Now suppose $\alpha_1, \beta_1$ are in $\bP^1(\F_{q^2})\setminus\mu_{q+1}$, and also $\alpha_2=\alpha_1^{-q}$ 
and $\beta_2=\beta_1^{-q}$. Define $\sigma(X):=(-\alpha_1^q X+1)/(X-\alpha_1)$ if $\alpha_1\ne\infty$, and 
$\sigma(X):=X$ otherwise, so that in any case $\sigma(X)$ permutes $\mu_{q+1}$ by Lemma~\ref{mutomu}, and 
also $\sigma(X)$ maps $\alpha_1$ and $\alpha_2$ to $\infty$ and $0$, respectively. Likewise define 
$\widetilde{\rho}(X):=(-\beta_1^q X+1)/(X-\beta_1)$ if $\beta_1\ne\infty$, and $\widetilde{\rho}(X):=X$ 
otherwise, so that $\widetilde{\rho}(X)$ permutes $\mu_{q+1}$ and maps $\beta_1$ and $\beta_2$ to $\infty$ 
and $0$, respectively. Thus $\widetilde g(X) := \widetilde{\rho}(X) \circ g(X) \circ \sigma^{-1}(X)$ is 
a degree-$n$ rational function in $\F_{q^2}(X)$ whose unique pole is $\infty$ and whose unique zero is $0$, 
whence $\widetilde g(X)=\epsilon X^n$ for some $\epsilon\in\F_{q^2}^*$. Since each of $\widetilde{\rho}(X)$, 
$g(X)$, and $\sigma^{-1}(X)$ maps $\mu_{q+1}$ into $\mu_{q+1}$, also $\widetilde g(\mu_{q+1})\subseteq\mu_{q+1}$, 
so that $\epsilon\in\mu_{q+1}$. Thus $g(X)$ permutes $\mu_{q+1}$ if and only if $\widetilde g(X)$ permutes 
$\mu_{q+1}$, or equivalently $\gcd(n,q+1)=1$, so that (1) is equivalent to each of (2) and (3). 
Finally, $\rho(X):=\epsilon^{-1}\widetilde{\rho}(X)$ permutes $\mu_{q+1}$, and we have 
$g(X)=\rho^{-1}\circ X^n\circ\sigma$, so that also (1) is equivalent to (4).
\end{proof}

We next recall an easy known fact about greatest common divisors, whose proof we include for the reader's convenience.

\begin{lemma}\label{gcd}
If $i$ and $j$ are positive integers then
\begin{itemize}
\item $\gcd(2^i+1,2^j+1)=1$ if and only if $\ord_2(i)\ne\ord_2(j)$;
\item $\gcd(2^i-1,2^j+1)=1$ if and only if $\ord_2(i)\le\ord_2(j)$.
\end{itemize}
\end{lemma}

\begin{proof}
It is well known that $\gcd(2^m-1,2^n-1)=2^{\gcd(m,n)}-1$ for any positive integers $m$ and $n$. Since 
$\gcd(2^j-1,2^j+1)=1$, it follows that
\[
\gcd(2^i-1,2^j+1)=\frac{\gcd(2^i-1,2^{2j}-1)}{\gcd(2^i-1,2^j-1)}=\frac{2^{\gcd(i,2j)}-1}{2^{\gcd(i,j)}-1},
\]
which equals $1$ precisely when $\gcd(i,2j)=\gcd(i,j)$, or equivalently $\ord_2(i)\le\ord_2(j)$. Likewise
\[
\gcd(2^i+1,2^j+1)=\frac{\gcd(2^{2i}-1,2^j+1)}{\gcd(2^i-1,2^j+1)}
=\frac{2^{\gcd(2i,2j)}-1}{2^{\gcd(2i,j)}-1}\cdot\frac{2^{\gcd(i,j)}-1}{2^{\gcd(i,2j)}-1},
\]
which equals $1$ precisely when $\ord_2(i)\ne\ord_2(j)$.
\end{proof}

We now apply the previous results to give a permutation criterion for the rational functions $g(X)$ 
considered in this paper.

\begin{prop}\label{key}
Write $q:=p^k$ and $Q:=p^\ell$ where $p$ is prime and $k$ and $\ell$ are positive integers. Let $a,b,c,d$ 
be elements of\/ $\F_{q^2}$ which are not all zero, and write $g(X):=B(X)/A(X)$ where $A(X):=aX^{Q+1}+bX^Q+cX+d$ 
and $B(X):=d^q X^{Q+1}+c^q X^Q+b^q X+a^q$. Then the following are equivalent:
\renewcommand{\labelenumi}{\emph{(\arabic{enumi})}}
\begin{enumerate}
\item $A(X)$ has no roots in $\mu_{q+1}$ and $g(X)$ permutes $\mu_{q+1}$.
\item $q$ is even, $g(X)=\rho^{-1}\circ X^n\circ\sigma$ for some degree-one $\rho,\sigma\in\F_{q^2}(X)$ 
such that $\rho(\mu_{q+1})=\sigma(\mu_{q+1})=\Omega$, and one of the following holds:
\renewcommand{\labelenumii}{\emph{(\alph{enumii})}}
\begin{enumerate}
\item $n=Q+1$, $\Omega=\bP^1(\F_q)$, and $\ord_2(k)\le\ord_2(\ell)$;
\item $n=Q+1$, $\Omega=\mu_{q+1}$, and $\ord_2(k)\ne\ord_2(\ell)$;
\item $n=Q-1$, $\Omega=\mu_{q+1}$, $\ord_2(k)\ge\ord_2(\ell)$, and there exists $\alpha\in\F_{q^2}$ such 
that $\sigma(\{\alpha,\alpha^{-q}\}) = \{0,\infty\}$ and $\gcd\bigl(A(X),B(X)\bigr)$ is a constant times 
$(\alpha^q X+1)(X+\alpha)$.
\end{enumerate}
\end{enumerate}
\end{prop}

\begin{proof}
First suppose that (2) holds. Since $\rho(X)$ and $\sigma(X)$ have degree $1$, they induce bijections 
on $\bP^1(\mybar\F_q)$. Thus $g(X)$ permutes $\mu_{q+1}$ if and only if $X^n$ permutes $\Omega$. Writing 
$\epsilon:=1$ if $\Omega=\mu_{q+1}$ and $\epsilon:=-1$ if $\Omega=\bP^1(\F_q)$, it follows that $g(X)$ permutes 
$\mu_{q+1}$ if and only if $\gcd(n,q+\epsilon)=1$, which holds in each of (a)--(c) by Lemma~\ref{gcd}. Write 
$C(X):= \gcd\bigl(A(X),B(X)\bigr)$, so that $A(X)/C(X)$ is a constant times the denominator of $g(X)$, which 
has no roots in $\mu_{q+1}$ since $g(X)$ permutes $\mu_{q+1}$. Thus $A(X)$ has no roots in $\mu_{q+1}$ if and 
only if $C(X)$ has no roots in $\mu_{q+1}$. In cases (a) and (b) we have $C(X) = 1$ since $\deg(g) = Q+1$ 
and $\max(\deg(A),\deg(B)) \le Q+1$. In case (c), by assumption $C(X)$ is a constant times 
$(\alpha^q X+1)(X+\alpha)$, which has no roots in $\mu_{q+1}$ since $\sigma(\{\alpha,\alpha^{-q}\}) = \{0,\infty\}$ 
and $\sigma$ permutes $\mu_{q+1}$. Thus $C(X)$ has no roots in $\mu_{q+1}$ in each of (a)--(c), so $A(X)$ has no 
roots in $\mu_{q+1}$, which implies (1).

Conversely, suppose henceforth that $A(X)$ has no roots in $\mu_{q+1}$ and $g(X)$ permutes $\mu_{q+1}$. 
Then in particular $g(X)$ is nonconstant, so part \eqref{A} of Theorem~\ref{geom} implies that either
\stepcounter{equation}
\renewcommand{\theenumi}{\theequation}
\renewcommand{\labelenumi}{(\theequation)}
\begin{enumerate}[leftmargin=\parindent,align=left,labelwidth=6pt,labelsep=6pt]
\item \label{A1repa}
$g(X)$ has ramification multiset $[1,Q]$ over some point in $\bP^1(\mybar\F_q)$
\end{enumerate}
or
\stepcounter{equation}
\begin{enumerate}[leftmargin=\parindent,align=left,labelwidth=6pt,labelsep=6pt]
\item \label{A2repa}
$g(X)=\rho^{-1}\circ X^n\circ\sigma$ for some $n\in\{Q-1,Q+1\}$ and some degree-one \\  
\phantom{x}\hspace{2ex}$\rho,\sigma\in\mybar\F_q(X)$.
\end{enumerate}
Corollary~\ref{lpp2} implies that \eqref{A1repa} cannot hold, so we must have \eqref{A2repa}. Since 
$B(X)=X^{Q+1}A^{(q)}(1/X)$, Lemma~\ref{cyclicperm} implies that $n$ is coprime to either $q-1$ or $q+1$, 
so since $n\in\{Q-1,Q+1\}$ it follows that $q$ is even. Thus all parts of (B) and (C) in Theorem~\ref{geom}
hold. By item (4) of Lemma~\ref{cyclicperm}, we may assume that $\rho,\sigma\in\F_{q^2}(X)$ and 
$\rho(\mu_{q+1})=\sigma(\mu_{q+1})=\Omega$ where either
\[
\gcd(n,q-1)=1 \text{ and } \Omega=\bP^1(\F_q)
\]
or
\[
\gcd(n,q+1)=1 \text{ and } \Omega=\mu_{q+1}.
\]
By Lemma~\ref{gcd}, if $n=Q+1$ then the former case yields (a) and the latter case yields (b). It remains only 
to show that if $n=Q-1$ then (c) holds.

Suppose $n=Q-1$. Then item \eqref{C4} of Theorem~\ref{geom} implies that the polynomial $U(X)$ defined in 
item \eqref{B2} of Theorem~\ref{geom} is a divisor of $A(X)$, so in particular $U(X)$ has no roots in $\mu_{q+1}$. 
Part \eqref{B} of Theorem~\ref{geom} implies that $U(X)$ has a degree-one term, and hence is nonconstant. Thus 
$U(X)$ has a root $\alpha\in \mybar\F_q$, which lies in $\F_{q^2}$ since either $U(X)$ is either SCR of degree 
$2$ or $U(X) = eX$ with $e \in \F_q^*$. Writing $\beta:= g(\alpha) \in \bP^1(\F_{q^2})$, part \eqref{C2} of 
Theorem~\ref{geom} yields $g^{-1}(\beta) = \{\alpha\}$, which implies $g^{-1}(\beta^{-q}) = \{\alpha^{-q}\}$ since 
$g^{(q)}(X) = X^{-1} \circ g(X) \circ X^{-1}$. Since $U(X)$ has no roots in $\mu_{q+1}$, we have $\alpha \notin \mu_{q+1}$, 
so that $\alpha^{-q} \ne \alpha$ and thus $\beta^{-q} \ne \beta$. Put $\widetilde{\sigma}(X):=(\alpha^q X+1)/(X+\alpha)$, 
and put $\widetilde\rho(X):=(\beta^q X+1)/(X+\beta)$ if $\beta\ne\infty$ and $\widetilde\rho(X):=X$ if $\beta=\infty$. 
Then $\widetilde{\sigma}(X)$ and $\widetilde\rho(X)$ permute $\mu_{q+1}$, and $\widetilde{\sigma}(X)$ maps $\alpha$ 
and $\alpha^{-q}$ to $\infty$ and $0$, respectively, while $\widetilde\rho(X)$ maps $\beta$ and $\beta^{-q}$ 
to $\infty$ and $0$, respectively. Thus $\widetilde g(X):=\widetilde\rho\circ g\circ\widetilde\sigma^{-1}$ 
is a degree-$(Q-1)$ rational function having $\infty$ and $0$ as its unique preimages of $\infty$ and $0$, 
respectively, so that $\widetilde g(X)=\gamma X^{Q-1}$ for some $\gamma\in\mybar\F_q^*$. Since both 
$\widetilde g(X)$ and $X^{Q-1}$ map $\mu_{q+1}$ into $\mu_{q+1}$, we must have $\gamma\in\mu_{q+1}$. Thus 
$g(X) = (\gamma^{-1} X \circ \widetilde{\rho})^{-1} \circ X^{Q-1}\circ \widetilde{\sigma}$. By replacing $\rho$ 
and $\sigma$ with $\gamma^{-1} X \circ \widetilde{\rho}$ and $\widetilde{\sigma}$ respectively, we get 
$g(X) = \rho^{-1} \circ X^{Q-1} \circ \sigma$ for some degree-one $\rho, \sigma \in \F_{q^2}(X)$ which permute 
$\mu_{q+1}$. Moreover, $\sigma(\{\alpha,\alpha^{-q}\}) = \{0,\infty\}$ by our construction, and Lemma~\ref{gcd} 
implies that $\ord_2(k)\ge\ord_2(\ell)$. Thus it remains only to show that $\gcd\bigl(A(X),B(X)\bigr)$ is 
a constant times $(\alpha^q X+1)(X+\alpha)$.

Item \eqref{C3} of Theorem~\ref{geom} implies that $C(X):= \gcd\bigl(A(X),B(X)\bigr)$ divides $U(X)$. Note that
the hypothesis $\deg(g)=Q-1$ implies $C(X)\ne 1$. If $\deg(U)=1$ then $U(X) = eX$ with $e\in \F_q^*$, so $C(X)=X$ 
and $\alpha=0$, whence $C(X)=(\alpha^q X+1)(X+\alpha)$. If $U(X)$ is a degree-$2$ SCR polynomial in $\F_{q^2}[X]$ 
then, since $\alpha\notin \mu_{q+1}$, we see that $\alpha^{-q}$ is another root of $U(X)$, so $U(X)$ is a constant 
times $(X+\alpha)(X+\alpha^{-q})$. Note that $B(X) = X^{Q+1} A^{(q)}(1/X)$ and $A(X) = X^{Q+1} B^{(q)}(1/X)$, 
so if one of $\alpha, \alpha^{-q}$ is a root of $C(X)$ then they both are roots of $C(X)$. Thus 
$C(X) = (X+\alpha)(X+\alpha^{-q})$. Therefore, regardless of whether $U(X)$ has degree $1$ or $2$, 
the polynomial $C(X)$ is a constant times $(\alpha^q X+1)(X+\alpha)$. This concludes the proof.
\end{proof}

\begin{prop}\label{mainpart}
Items \emph{(1)} and \emph{(3)} in Theorem~\emph{\ref{main}} are equivalent to one another.
\end{prop}

\begin{proof}
By Proposition~\ref{key}, item (1) of Theorem~\ref{main} holds if and only if item (2) of Proposition~\ref{key} holds.
It remains only to show that the polynomials $A(X)$ corresponding to cases (a), (b) and (c) of Proposition~\ref{key} 
are precisely the polynomials $A(X)$ satisfying items (1), (2), (3) of Theorem~\ref{fir2}, respectively.

First consider case (a). By Lemma~\ref{mutoF}, the degree-one $\rho,\sigma\in\F_{q^2}(X)$ which map $\mu_{q+1}$ 
to $\bP^1(\F_q)$ are
\[
\sigma(X):=\frac{\alpha X+\alpha ^q}{X+1}\circ \gamma X \quad\text{ and } \rho(X):=\frac{\beta X+\beta ^q}{X+1}\circ \lambda X
\]
with $\alpha,\beta\in\F_{q^2}\setminus\F_q$ and $\gamma,\lambda\in\mu_{q+1}$. Then
\[
\lambda X\circ \rho^{-1}\circ X^{Q+1}\circ\sigma\circ \gamma^{-1} X = 
\frac{(\alpha X+\alpha ^q)^{Q+1} + \beta ^q (X+1)^{Q+1}}{(\alpha X+\alpha ^q)^{Q+1} + \beta (X+1)^{Q+1}},
\]
whose denominator is an element of $\F_{q^2}^*$ times
\[
(\alpha^{Q+1}+\beta)X^{Q+1} + (\alpha^{q+Q}+\beta)X^Q + (\alpha^{qQ+1}+\beta)X + (\alpha^{qQ+q}+\beta).
\]
Thus case (a) of Proposition~\ref{key} yields precisely the polynomials $A(X)=\delta A_0(\gamma X)$ where 
$\gamma\in\mu_{q+1}$, $\delta\in\F_{q^2}^*$,
and $k,\ell,A_0(X)$ 
satisfy case (1) of Theorem~\ref{fir2}.

Next consider case (b). By Lemma~\ref{mutomu}, the degree-one rational functions in $\F_{q^2}(X)$ which 
permute $\mu_{q+1}$ are $\gamma X$ and $(\alpha^q X+1)/(X+\alpha)\circ \gamma X$ with $\gamma\in\mu_{q+1}$ and 
$\alpha\in\F_{q^2}\setminus\mu_{q+1}$.
%
%
Now put $g(X):=\rho^{-1}\circ X^{Q+1}\circ\sigma$ where $\rho,\sigma\in\F_{q^2}(X)$ are degree-one rational 
functions permuting $\mu_{q+1}$. By replacing $\rho(X)$ and $\sigma(X)$ by $1/\rho(X)$ and $1/\sigma(X)$ if 
necessary, we may assume that $\sigma(0)\ne 0$, and that if $\sigma(0)\ne\infty$ then $\rho(0)\ne 0$. Then 
the possibilities for $g(X)$, up to replacing $g(X)$ by $\lambda g(\gamma^{-1} X)$ with $\gamma,\lambda\in\mu_{q+1}$, 
are $1/X^{Q+1}$ and
\[
\frac{\beta^q(\alpha^q X+1)^{Q+1}+(X+\alpha)^{Q+1}}{(\alpha^q X+1)^{Q+1}+\beta (X+\alpha)^{Q+1}}
\]
with $\alpha,\beta\in\F_{q^2}\setminus\mu_{q+1}$. Thus case (b) of Proposition~\ref{key} yields precisely the 
polynomials $A(X)=\delta A_0(\gamma X)$ where $\gamma\in\mu_{q+1}$, $\delta\in\F_{q^2}^*$, and $k,\ell,A_0(X)$ 
satisfy case (2) of Theorem~\ref{fir2}.

Finally, consider case (c), so that $\ord_2(k)\ge\ord_2(\ell)$. The rational functions $g(X)$ in this case are 
$g(X)=\rho^{-1}\circ X^{Q-1}\circ\sigma$ for degree-one $\rho,\sigma\in\F_{q^2}(X)$ which permute $\mu_{q+1}$ and 
satisfy $\sigma(\{\widetilde{\alpha},\widetilde{\alpha}^{-q}\}) = \{0,\infty\}$ for some $\widetilde\alpha\in\F_{q^2}$.
By replacing $\rho(X)$ and $\sigma(X)$ by $1/\rho(X)$ and $1/\sigma(X)$ if necessary, we may assume that either 
$\rho(\infty),\sigma(\infty)\ne \infty$ or both $\rho(\infty)= \infty$ and $\sigma(\infty)=0$.

The degree-one $\rho,\sigma\in\F_{q^2}(X)$ which permute $\mu_{q+1}$ and satisfy $\rho(\infty)=\infty$ and 
$\sigma(\infty)=0$ are $\rho(X) = X/\lambda$ and $\sigma(X) = \gamma/X$ with $\gamma,\lambda\in \mu_{q+1}$. 
Here $\sigma^{-1}(\{\infty,0\}) = \{\infty,0\}$, which equals $\{\widetilde{\alpha},\widetilde{\alpha}^{-q}\}$ 
with $\widetilde\alpha\in\F_{q^2}$ if and only if $\widetilde\alpha=0$. Thus the possibilities for $g(X)$ when 
$\rho(\infty)=\infty$ and $\sigma(\infty)=0$ are
\[
\lambda X\circ\frac{1}{X^{Q-1}}\circ\frac{X}{\gamma} = \frac{\lambda\gamma^{Q-1}}{X^{Q-1}}.
\]
It follows that the possibilities for $A,B\in\F_{q^2}[X]$ where $g(X)=B(X)/A(X)$ and 
$\gcd(A(X),B(X))\in\F_{q^2}^*\cdot(\widetilde\alpha^q X+1)(X+\widetilde\alpha)$ are precisely $A(X)=\delta X^Q$ 
and $B(X)=\delta\lambda\gamma^{Q-1} X$ with $\delta\in\F_{q^2}^*$, so that $B(X)=X^{Q+1}A^{(q)}(1/X)$ if and only 
if $\delta^{q-1}=\lambda\gamma^{Q-1}$. Thus the possibilities for $A(X)$ when $B(X)=X^{Q+1}A^{(q)}(1/X)$ are 
$\delta X^Q$ with $\delta\in\F_{q^2}^*$.

The degree-one rational functions $\rho,\sigma\in\F_{q^2}(X)$ which permute $\mu_{q+1}$ and satisfy 
$\rho(\infty),\sigma(\infty)\ne \infty$ are $\rho(X) = (\beta X+1)/(X+\beta^q)\circ \lambda^{-1} X$ 
and $\sigma(X) = (\alpha^q X+1)/(X+\alpha)\circ \gamma X$ with $\gamma,\lambda\in\mu_{q+1}$ 
and $\alpha,\beta\in\F_{q^2}\setminus\mu_{q+1}$. Thus 
$\{\gamma^{-1}\alpha,\gamma^{-1}\alpha^{-q}\} = \sigma^{-1}(\{\infty,0\})$,
which equals $\{\widetilde\alpha,\widetilde\alpha^{-q}\}$ if and only if 
$\widetilde\alpha\in\{\gamma^{-1}\alpha,\gamma^{-1}\alpha^{-q}\}$. The possibilities for $g(X)$ in this case are
\[
\lambda X \circ 
\frac{\beta^q(\alpha^q X+1)^{Q-1}+(X+\alpha)^{Q-1}}{(\alpha^q X+1)^{Q-1}+\beta (X+\alpha)^{Q-1}} \circ \gamma X.
\]
It follows that the possibilities for $A,B\in\F_{q^2}[X]$ where $g(X)=B(X)/A(X)$ and 
$\gcd(A(X),B(X))\in\F_{q^2}^*\cdot(\widetilde\alpha^q X+1)(X+\widetilde\alpha)$ are
\begin{align*}
A(X)&=\delta(\gamma X+\alpha)(\alpha^q\gamma X+1)\cdot\Bigl( (\alpha^q\gamma X+1)^{Q-1}+\beta (\gamma X+\alpha)^{Q-1} \Bigr) \\
&=\delta\Bigl((\gamma X+\alpha)(\alpha^q\gamma X+1)^Q + \beta (\alpha^q\gamma X+1)(\gamma X+\alpha)^Q\Bigr)
\end{align*}
and
\[
B(X)=\delta\lambda\Bigl(\beta^q(\gamma X+\alpha)(\alpha^q \gamma X+1)^Q + (\alpha^q\gamma X+1)(\gamma X+\alpha)^Q\Bigr)
\]
with $\delta\in\F_{q^2}^*$. Here $B(X)=X^{Q+1}A^{(q)}(1/X)$ if and only if $\delta^{q-1}=\lambda\gamma^{Q+1}$, so the 
possibilities for $A(X)$ when $B(X)=X^{Q+1}A^{(q)}(1/X)$ are precisely the polynomials $\delta A_0(\gamma X)$ with 
$\delta\in\F_{q^2}^*$, $\gamma\in\mu_{q+1}$, and $A_0(X)$ occurring in case (3) of Theorem~\ref{fir2} with $A_0(X)\ne X^Q$.
\end{proof}

We conclude this section with a proof of Theorem~\ref{fir2}.

\begin{proof}[Proof of Theorem~\ref{fir2}]
By Lemma~\ref{old} and Lemma~\ref{rewrite}, $X^r A(X^{q-1})$ permutes $\F_{q^2}$ if and only if $\gcd(r,q-1)=1$, $A(X)$ 
has no roots in $\mu_{q+1}$, and $g(X):=X^{Q+1} A^{(q)}(1/X)/A(X)$ permutes $\mu_{q+1}$. Now the result follows from 
Proposition~\ref{mainpart}.
\end{proof}


\section{Restating conditions in terms of coefficients}
\label{sec:Q}

In this section we give conditions on the coefficients of $A(X)$ which are equivalent to the condition $U(X)\mid A(X)$ 
occurring in item \eqref{B2} of Theorem~\ref{geom}. We will use these conditions in the next section to prove Theorem~\ref{fir} 
and Theorem~\ref{main}. We use the following notation throughout this section:
\begin{itemize}
\item $q=2^k$ and $Q=2^\ell$ where $k$ and $\ell$ are positive integers,
\item $m:=2^{\ord_2(\gcd(k,\ell))}$,
\item $a,b,c,d$ are elements of $\F_{q^2}$ for which $e:=a^{q+1}+b^{q+1}+c^{q+1}+d^{q+1}$ is nonzero and
\begin{equation}\label{u/v constant}
(ab^q+cd^q)^Q=e^{Q-1}(ac^q+bd^q),
\end{equation}
\item $A(X):=aX^{Q+1}+bX^Q+cX+d$,
\item $U(X):=(ab^q+cd^q)X^2+eX+a^qb+c^qd$,
\item if $s$ is a power of $2$ and $t$ is a power of $s$ then we define $\Tr_{\F_t/\F_s}(X)$ to be the polynomial 
$X+X^s+X^{s^2}+X^{s^3}+\dots+X^{t/s}$, 
\vspace{0.12cm}
\item $\zeta:=\displaystyle{\frac{(ab^q+cd^q)^{q+1}}{e^2}}$,\, $\eta:=\displaystyle{\frac{b^{q+1}+c^{q+1}}e}$,\, 
and $\theta:=\eta+\Tr_{\F_Q/\F_2}(\zeta)$.
\end{itemize}


\begin{lemma}\label{u|A}
Suppose that $\{b,c,d\}\ne\{0\}$. Then $U(X)\mid A(X)$ if and only if $\theta = 1$.
%
%
\end{lemma}

\begin{proof}
First suppose $\deg(U)\ne 2$. Then $U(X)=eX$ and $ab^q=cd^q$, so by \eqref{u/v constant} we have $ac^q=bd^q$. Here 
$\zeta=0$, so that $\eta+\Tr_{\F_Q/\F_2}(\zeta)$ equals $1$ if and only if $a^{q+1}=d^{q+1}$. If $a^{q+1}=d^{q+1}$ then 
since $e\ne 0$ we have $b^{q+1}\ne c^{q+1}$; since the $(q+1)$-th power of $ab^q=cd^q$ implies that $(ab)^{q+1}=(cd)^{q+1}$, 
it follows that $d=0$. Conversely, if $d=0$ then since $\{b,c\}\ne\{0\}$ and $ab^q=0=ac^q$ we conclude that $a=0$. 
Thus $a^{q+1}=d^{q+1}$ if and only if $d=0$. This yields the desired conclusion, since plainly $d=0$ if and only if 
$U(X)$ divides $A(X)$.

Henceforth suppose $\deg(U)=2$. Then $ab^q\ne cd^q$, which by \eqref{u/v constant} implies that $ac^q\ne bd^q$. The 
hypothesis $e\ne 0$ implies that $U(X)$ has two distinct roots in $\mybar\F_q$. Thus $U(X)\mid A(X)$ if and only if 
each root $\alpha$ of $U(X)$ satisfies $A(\alpha)=0$. Writing
\[
\beta:=\frac{ab^q+cd^q}e\alpha,
\]
the condition $U(\alpha)=0$ says that $\beta^2+\beta=\zeta$. Applying $\Tr_{\F_Q/\F_2}(X)$ to both sides yields
\[
\beta^Q + \beta = \Tr_{\F_Q/\F_2}(\zeta), 
\]
or equivalently
\[
\alpha^Q \frac{(ab^q+cd^q)^Q}{e^Q} + \alpha \frac{ab^q+cd^q}e = \Tr_{\F_Q/\F_2}(\zeta).
\]
By \eqref{u/v constant} we have
\[
\frac{(ab^q+cd^q)^Q}{e^Q}=\frac{ac^q+bd^q}e.
\]
Now
\[
A(\alpha) = (a\alpha+b) \alpha^Q + c\alpha+d
\]
so that
\[
A(\alpha) \frac{ac^q+bd^q}e = (a\alpha+b) \Bigl( \alpha \frac{ab^q+cd^q}e + \Tr_{\F_Q/\F_2}(\zeta) \Bigr) + (c\alpha+d) \frac{ac^q+bd^q}e.
\]
Since $ac^q\ne bd^q$, it follows that $A(\alpha)=0$ if and only if $P(\alpha)=0$, where
\begin{align*}
P(X) &:= X^2 a \frac{ab^q+cd^q}e + X \Bigl( b \frac{ab^q+cd^q}e+c\frac{ac^q+bd^q}e+a\Tr_{\F_Q/\F_2}(x)\Bigr) \\
&\qquad + d \frac{ac^q+bd^q}e + b \Tr_{\F_Q/\F_2}(\zeta).
\end{align*}
Since $U(X)$ has two distinct roots, and $\deg(U)\ge\deg(P)$, it follows that $U(X)\mid A(X)$ if and only if 
$P(X)=\gamma U(X)$ for some $\gamma\in\F_{q^2}$. By comparing degree-$2$ coefficients, we see that $\gamma=a/e$. 
Thus $U(X)\mid A(X)$ if and only if both
\[
a = b \frac{ab^q+cd^q}e + c \frac{ac^q+bd^q}e + a \Tr_{\F_Q/\F_2}(\zeta)
\]
and
\[
a\frac{a^q b + c^q d}e = d \frac{ac^q+bd^q}e + b \Tr_{\F_Q/\F_2}(\zeta),
\]
or equivalently both
\[
a(1+\eta)=a\Tr_{\F_Q/\F_2}(\zeta)
\]
%
%
and
\[
b(1+\eta)=b\Tr_{\F_Q/\F_2}(\zeta).
\]
%
%
At least one of $a$ or $b$ is nonzero since $ac^q\ne bd^q$, so the combination of the above two equations is equivalent to
\[
1+\eta=\Tr_{\F_Q/\F_2}(\zeta),
\]
%
%
as desired.
\end{proof}

In the rest of this section we prove further results about $\zeta$, $\eta$, and $\theta$. The first step is the following 
result which determines the possible values of $\theta$.

\begin{lemma}\label{m in F_2}
The element $\theta$
%
%
is in\/ $\F_2$.
\end{lemma}

\begin{proof}
By \eqref{u/v constant} we have
\[
e^2 \zeta^Q = \frac{(ab^q+cd^q)^{Q(q+1)}}{e^{2Q-2}} = \frac{(ac^q+bd^q)^{q+1}}{e^{2Q-2-(Q-1)(q+1)}} = (ac^q+bd^q)^{q+1},
\]
where the last equality holds because $e\in\F_q$. It is easy to check that
\[
(ac^q+bd^q)^{q+1} + (ab^q+cd^q)^{q+1} = (a^{q+1}+d^{q+1})\cdot(b^{q+1}+c^{q+1}),
\]
so that
\begin{align*}
e^2(\theta^2&+\theta) = e^2 (\zeta^Q+\zeta+\eta^2+\eta) \\
&= (ac^q+bd^q)^{q+1}+(ab^q+cd^q)^{q+1}+(b^{q+1}+c^{q+1})^2+e(b^{q+1}+c^{q+1}) \\
&= (a^{q+1}+d^{q+1}+b^{q+1}+c^{q+1}+e)\cdot (b^{q+1}+c^{q+1}) \\
&= 0.
\end{align*}
It follows that $\theta^2+\theta=0$ and thus $\theta\in\F_2$.
\end{proof}

\begin{lemma}\label{compose traces}
Pick $\alpha \in \F_q$ and write $\beta:=\Tr_{\F_q/\F_{2^m}} ( \Tr_{\F_Q/\F_2}(\alpha) )$.
\renewcommand{\labelenumi}{\emph{(\arabic{enumi})}}
\begin{enumerate}
\item
If $\ord_2(k) < \ord_2(\ell)$ then $\beta=0$.
\item If $\ord_2(k) \ge \ord_2(\ell)$ then $\beta=\Tr_{\F_q/\F_2}(\alpha)$.
\end{enumerate}
\end{lemma}

\begin{proof}
Since the polynomials $\Tr_{\F_q/\F_{2^m}}(X)$ and $\Tr_{\F_Q/\F_2}(X)$ commute under composition, we have 
$\beta=\Tr_{\F_Q/\F_2}(\gamma)$ where $\gamma:=\Tr_{\F_q/\F_{2^m}}(\alpha)$. Since $\gamma\in\F_{2^m}$, we have
\[
\beta=\Tr_{\F_{2^m}/\F_2}\bigl(\Tr_{\F_Q/\F_{2^m}}(\gamma)\bigr)=\Tr_{\F_{2^m}/\F_2}\Bigl(\frac{\ell}{m}\gamma\Bigr).
\]
Thus if $\ord_2(k) < \ord_2(\ell)$ then $\beta=0$, and if $\ord_2(k) \ge \ord_2(\ell)$ then
\[
\beta = \Tr_{\F_{2^m}/\F_2}(\gamma) = \Tr_{\F_{2^m}/\F_2}\bigl(\Tr_{\F_q/\F_{2^m}}(\alpha)\bigr)=\Tr_{\F_q/\F_2}(\alpha). \qedhere
\]
\end{proof}

The final result in this section gives a connection between traces in different field extensions, which will be useful 
when we apply Lemma~\ref{u|A}.

\begin{lemma}\label{traces}
We have
\[
\Tr_{\F_q/\F_{2^m}}(\eta) =
\begin{cases}
\theta & \text{if $\ord_2(k) < \ord_2(\ell)$}, \\
\frac{k}m \theta + \Tr_{\F_q/\F_2}(\zeta) & \text{if $\ord_2(k) \ge \ord_2(\ell)$.}
\end{cases}
\]
%
%
\end{lemma}

\begin{proof}
The definition of $\theta$ says that $\Tr_{\F_Q/\F_2}(\zeta)=\theta+\eta$.
By applying $\Tr_{\F_q/\F_{2^m}}(X)$ to both sides, and noting that $\theta \in\F_2$ by Lemma~\ref{m in F_2}, we obtain
\[
\Tr_{\F_q/\F_{2^m}}\bigl(\Tr_{\F_Q/\F_2}(\zeta)\bigr)=\frac{k}m \theta+\Tr_{\F_q/\F_{2^m}}(\eta).
\]
Note that $\zeta$ is in $\F_q$, so Lemma~\ref{compose traces} implies that if $\ord_2(k) < \ord_2(\ell)$ then
\[
\Tr_{\F_q/\F_{2^m}} ( \Tr_{\F_Q/\F_2}(\zeta) ) = 0,
\] 
and if $\ord_2(k) \ge \ord_2(\ell)$ then 
\[ 
\Tr_{\F_q/\F_{2^m}} \bigl( \Tr_{\F_Q/\F_2}(\zeta) \bigr) = \Tr_{\F_q/\F_2}(\zeta).
\]
The result follows.
\end{proof}


\section{Proof of Theorems~\ref{fir} and \ref{main}}
\label{sec:main}

In this section we prove Theorem~\ref{main}, which implies Theorem~\ref{fir} in light of Lemmas~\ref{old} and \ref{rewrite}.

\begin{proof}[Proof of Theorem~\ref{main}]
By Proposition~\ref{mainpart}, items (1) and (3) of Theorem~\ref{main} are equivalent, so we need only show that items 
(1) and (2) are equivalent. We use the following notation:
\begin{itemize}
\item $q:=p^k$ and $Q:=p^\ell$ where $p$ is prime and $k,\ell>0$,
\item $m:=p^{\ord_2(\gcd(k,\ell))}$,
\item $a,b,c,d\in\F_{q^2}$ are not all zero,
\item $A(X):=aX^{Q+1}+bX^Q+cX+d$,
\item $B(X):=d^q X^{Q+1}+c^q X^Q+b^q X+a^q$,
\item $g(X):=B(X)/A(X)$.
\end{itemize}

First suppose that
\stepcounter{equation}
\renewcommand{\theenumi}{\theequation}
\renewcommand{\labelenumi}{(\theequation)}
\begin{enumerate}[leftmargin=\parindent,align=left,labelwidth=6pt,labelsep=6pt]
\item \label{nortsrep} $A(X)$ has no roots in $\mu_{q+1}$
\end{enumerate}
and
\stepcounter{equation}
\begin{enumerate}[leftmargin=\parindent,align=left,labelwidth=6pt,labelsep=6pt]
\item \label{permrep} $g(X)$ permutes $\mu_{q+1}$.
\end{enumerate}
By Proposition~\ref{key} it follows that $q$ is even and
\stepcounter{equation}
\begin{enumerate}[leftmargin=\parindent,align=left,labelwidth=6pt,labelsep=6pt]
\item \label{A2rep}
$g(X)=\rho\circ X^n\circ\sigma$ for some $n\in\{Q-1,Q+1\}$ and some degree-one \\  \phantom{x}\hspace{2ex}$\rho,\sigma\in\mybar\F_q(X)$.
\end{enumerate}
Then part \eqref{B} of Theorem~\ref{geom} implies that both of the following hold:
\begin{gather}
e:=a^{q+1}+b^{q+1}+c^{q+1}+d^{q+1} \text{ is nonzero}\label{e nonzero}, \\
(ab^q+cd^q)^Q=e^{Q-1}(ac^q+bd^q).\label{u/v constant rep}
\end{gather}

Conversely, we assume henceforth that $q$ is even and both \eqref{e nonzero} and \eqref{u/v constant rep}  hold. 
It remains only to show that the combination of \eqref{nortsrep}, \eqref{permrep}, and \eqref{A2rep} is equivalent to
\begin{equation} \label{trrep}
\Tr_{\F_q/\F_{2^m}}\Bigl(\frac{b^{q+1}+c^{q+1}}e\Bigr)=\frac{\lcm(k,\ell)}m.
\end{equation}
By part \eqref{B} of Theorem~\ref{geom}, the combination of \eqref{nortsrep} and \eqref{A2rep} is equivalent to
\stepcounter{equation}
\renewcommand{\theenumi}{\theequation}
\renewcommand{\labelenumi}{(\theequation)}
\begin{enumerate}[leftmargin=\parindent,align=left,labelwidth=6pt,labelsep=6pt]
\item \label{urep} either $U(X)\nmid A(X)$ or $U(X)$ has no roots in $\mu_{q+1}$, where 
\[ U(X):=(ab^q+cd^q)X^2+eX+(ab^q+cd^q)^q.\]
\end{enumerate}
Write
\begin{align*}
\zeta&:=\frac{(ab^q+cd^q)^{q+1}}{e^2}, \\
\eta&:=\frac{b^{q+1}+c^{q+1}}e, \\
\theta&:=\eta+\Tr_{\F_Q/\F_2}(\zeta), \\
W(X)&:=(bc+ad)X^2+eX+(bc+ad)^q,
\end{align*}
where we maintain the convention from the previous section that $\Tr_{\F_Q/\F_2}(X)$ denotes the polynomial $X^{Q/2}+X^{Q/4}+\dots+X$, 
and hence can be evaluated at elements of $\mybar\F_q$. Let $\Lambda$ be the union of the set of roots of $W(X)$ in $\mybar\F_q$ 
and the set consisting of $(2-\deg(W))$ copies of $\infty$. Since $e\ne 0$, we know that $W(X)$ has $\deg(W)$ roots, so that 
$\abs{\Lambda}=2$.

Suppose $\{b,c,d\}=\{0\}$. Then \eqref{urep} holds since $U(X)=eX$ has no roots in $\mu_{q+1}$, so we must show that \eqref{trrep} 
is equivalent to \eqref{permrep}. Here \eqref{trrep} holds if and only if $\ord_2(k)\ne\ord_2(\ell)$, which by Lemma~\ref{gcd} 
says that $\gcd(Q+1,q+1)=1$, or equivalently $g(X)=a^{q-1}/X^{Q+1}$ permutes $\mu_{q+1}$.

Now suppose that $\{b,c,d\}\ne\{0\}$ and \eqref{urep} does not hold. Then $U(X)\mid A(X)$ and $U(X)$ has roots in $\mu_{q+1}$, 
which by Lemmas~\ref{u|A} and \ref{quadscr} implies that $\theta=1$ and $\Tr_{\F_q/\F_2}(\zeta)=1$. Now Lemma~\ref{traces} shows 
that if $\ord_2(k) < \ord_2(\ell)$ then $\Tr_{\F_q/\F_{2^m}}(\eta)=1\ne\lcm(k,\ell)/m$, and if $\ord_2(k) \ge \ord_2(\ell)$ then 
$\Tr_{\F_q/\F_{2^m}}(\eta)=1+k/m\ne\lcm(k,\ell)/m$. In either case, \eqref{trrep} does not hold, so that both \eqref{trrep} 
and \eqref{urep} do not hold.

Henceforth suppose that $\{b,c,d\}\ne\{0\}$ and \eqref{urep} holds. Thus also \eqref{nortsrep} and \eqref{A2rep} hold, so it remains 
to show that \eqref{trrep} is equivalent to \eqref{permrep}. Here either $U(X)\nmid A(X)$ or $U(X)$ has no roots in $\mu_{q+1}$, and 
part \eqref{C} of Theorem~\ref{geom} shows that
\stepcounter{equation}
\begin{enumerate}[leftmargin=\parindent,align=left,labelwidth=6pt,labelsep=6pt,itemsep=4pt]
\item \label{Rprerep} each element of $\Lambda$ has a unique $g$-preimage in $\bP^1(\mybar\F_q)$,
\stepcounter{equation}
\item\label{Ru} each root of $U(X)$ in $\bP^1(\mybar\F_q)$ is the unique $g$-preimage of some $\beta\in\Lambda$,
\stepcounter{equation}
\item\label{degrep} $n=Q-1$ if and only if $U(X)\mid A(X)$.
\end{enumerate}
Lemmas~\ref{u|A} and \ref{m in F_2} show that $\theta=1$ if $U(X)\mid A(X)$ and $\theta=0$ otherwise, so by \eqref{degrep} we have 
$n=Q-1$ if $\theta=1$ and $n=Q+1$ if $\theta=0$.

Suppose that $n=Q-1$, so that $\theta=1$ and $U(X)\mid A(X)$. Then \eqref{urep} implies that $U(X)$ has no roots in $\mu_{q+1}$, which 
by Lemma~\ref{quadscr} yields $\Tr_{\F_q/\F_2}(\zeta)=0$. By Lemma~\ref{traces}, if $\ord_2(k) < \ord_2(\ell)$ then $\Tr_{\F_q/\F_{2^m}}(\eta)=1$, 
and if $\ord_2(k) \ge \ord_2(\ell)$ then $\Tr_{\F_q/\F_{2^m}}(\eta)=k/m$. Thus \eqref{trrep} holds if and only if $\ord_2(k) \ge \ord_2(\ell)$, 
which by Lemma~\ref{gcd} is equivalent to $\gcd(Q-1,q+1)=1$. Since $U(X)$ has no roots in $\mu_{q+1}$, by \eqref{Ru} and 
Lemma~\ref{cyclicperm} we see that \eqref{trrep} and \eqref{permrep} are equivalent.

The remaining possibility is $n=Q+1$, so that $\theta=0$ and $U(X)\nmid A(X)$. By Lemma~\ref{traces}, if $\ord_2(k) < \ord_2(\ell)$ 
then $\Tr_{\F_q/\F_{2^m}}(\eta)=0$, and if $\ord_2(k) \ge \ord_2(\ell)$ then $\Tr_{\F_q/\F_{2^m}}(\eta)=\Tr_{\F_q/\F_2}(\zeta)$.

If $U(X)$ has roots in $\mu_{q+1}$ then $\Tr_{\F_q/\F_2}(\zeta)=1$, and Lemma~\ref{cyclicperm} implies that \eqref{permrep} holds if and only 
if $\gcd(Q+1,q-1)=1$. By Lemma~\ref{gcd} this is equivalent to $\ord_2(k) \le \ord_2(\ell)$, which by the previous paragraph is equivalent 
to \eqref{trrep}.

Finally, if $U(X)$ has no roots in $\mu_{q+1}$ then $\Tr_{\F_q/\F_2}(\zeta)=0$, and Lemma~\ref{cyclicperm} implies that \eqref{permrep} holds 
if and only if $\gcd(Q+1,q+1)=1$. By Lemma~\ref{gcd} this is equivalent to $\ord_2(k) \ne \ord_2(\ell)$, which as above is equivalent 
to \eqref{trrep}.
\end{proof}


\section{A general polynomial identity}
\label{sec:biv}

In this section we prove an identity involving bivariate polynomials. This identity seems interesting for its own sake, 
and in addition we use it in the next section to prove the conjecture from \cite[Rem.~2]{LHXZ} about permutations of
$\F_q\times\F_q$. Unexpectedly, this identity provides a new proof of a result of Cusick and M\"uller \cite{CM} about 
the functions $\F_q\to\F_q$ induced by certain polynomials that are \emph{not} permutation polynomials.

Throughout this section we use the following notation:

\begin{itemize}
\item $q$ is a prime power,
\item $n$ is a positive integer,
\item $\Delta$ is the set of all elements $(w^q-w)^{q-1}$ with $w \in \F_{q^n}\setminus\F_q$.
\end{itemize}

The main result in this section is as follows.

\begin{thm}\label{identity}
The following identity holds in\/ $\F_{q^n}[X,Y]$:
\begin{equation}\label{apeq}
\begin{multlined}
  X^{q^n-1}-1 + \prod_{z \in \F_{q^n}^*} ( 1 + z X - z^q Y )  \\
=  -Y \cdot  \biggl( \frac{ X^{q^n-1} - Y^{q^n-1} } { X^{\frac{q^n-1}{q-1}} - Y^{\frac{q^n-1}{q-1}} } \biggr) 
 \cdot \Big( Y^{\frac{q^n-q}{q-1}} + \sum_{i=1}^{n-1} X^{1+\frac{q^n-q^{i+1}}{q-1}} Y^{\frac{q^i-q}{q-1}} \Big) .
\end{multlined}
\end{equation}
\end{thm}

We first factor the left side
of \eqref{apeq} in case $X=1$.

\begin{lemma} \label{1st}
We have $\abs{\Delta}=(q^{n-1}-1)/(q-1)$ and
\[
\prod_{z \in \F_{q^n}^*} ( 1 + z - z^q Y ) = - \biggl( \frac{ Y^{q^n} - Y } { Y^{\frac{q^n-1}{q-1}} - 1 } \biggr) \cdot \prod_{u \in \Delta} (Y - u)^q.
\]
\end{lemma}

\begin{proof}
For any $u\in \F_{q^n}$, let $\phi_u \colon \F_{q^n} \to \F_{q^n}$ be the function $z\mapsto u z^q-z$, and note that 
$\phi_u$ is a linear map of $\F_q$-vector spaces. For any $z\in\F_{q^n}^*$, the unique root of $1+z-z^q Y$ is the 
unique $u\in\F_{q^n}$ for which $\phi_u(z)=1$. Thus the multiplicity of any prescribed $u\in\F_{q^n}$ as a root of 
$\prod_{z\in\F_{q^n}^*} (1+z-z^q Y)$ equals $e_u:=\abs{\phi_u^{-1}(1)}$. We now compute $e_u$ for each $u\in\F_{q^n}$.
This is easy when $u$ is not in $\Omega:=\mu_{\frac{q^n-1}{q-1}}$, since then $\phi_u$ has trivial kernel and hence is 
bijective, so that $e_u=1$. Now suppose that $u\in \Omega$, so that $u=v^{q-1}$ for some $v\in\F_{q^n}^*$. Then the 
kernel of $\phi_u$ is $\{\alpha/v\colon\alpha\in\F_q\}$, whence $e_u$ is in $\{0,q\}$, with $e_u=q$ if and only 
if $uz^q-z=1$ for some $z\in\F_{q^n}$. Write $w:=vz$ to restate the last condition as $w^q-w=v$ for some $w\in\F_{q^n}$. 
Since $v\ne 0$, the last condition occurs if and only if $u\in \Delta$. Combining the above conclusions with the 
equality $\prod_{z\in \F_{q^n}^*} z = -1$ yields
\begin{align*}
\prod_{z \in \F_{q^n}^*} ( 1 + z - z^q Y )
&= - \prod_{u \in \F_{q^n}} ( Y - u )^{e_u} \\
&= - \prod_{u\in \F_{q^n}\setminus \Omega} ( Y - u ) \cdot \prod_{u\in \Delta} ( Y - u )^q \\ 
&= - \biggl( \frac{ Y^{q^n} - Y } { Y^{\frac{q^n-1}{q-1}} - 1 } \biggr) \cdot \prod_{u \in \Delta} (Y - u)^q.
\end{align*}
The expression for $\abs{\Delta}$ follows upon equating the degrees of the two sides.
\end{proof}

Lemma~\ref{1st} yields the following reformulation of the main result of \cite{CM}.

\begin{cor}\label{RM}
The following identity holds in\/ $\F_{q^n}[Y]$:
\[
\prod_{z\in \F_{q^n}} \bigl( Y - (z+1)z^{q-1} \bigr) = Y^2 \cdot \biggl( \frac{Y^{q^n-1}-1}{Y^{\frac{q^n-1}{q-1}}-1} \biggr) \cdot \prod_{u\in \Delta} (Y-u)^q.
\]
\end{cor}

\begin{proof}
This follows from Lemma~\ref{1st} since
\begin{align*}
\prod_{z\in \F_{q^n}^*} \bigl( Y - (z+1)z^{q-1} \bigr)&=
\prod_{z\in\F_{q^n}^*} \bigl( Y - z^{q-1} - z^q \bigr) \\
&=\prod_{w\in\F_{q^n}^*} \bigl( Y - w^{1-q} - w^{-q} \bigr) \\ &= \prod_{w\in\F_{q^n}^*}\bigl(w^{-q}\cdot( w^q Y - w - 1 )\bigr) \\
&= -\prod_{w\in\F_{q^n}^*}\bigl(w^q Y - w - 1\bigr).\qedhere
\end{align*}
\end{proof}

\begin{lemma}\label{2nd}
The following identity holds in\/ $\F_{q^n}[Y]$:
\[
\prod_{u\in \Delta} (Y-u) = \sum_{i=1}^n Y^{\frac{q^{i-1}-1}{q-1}}.
\]
\end{lemma}

\begin{proof}
Since both sides are monic polynomials in $\F_{q^n}[Y]$ of the same degree, it suffices to show that the left side divides 
the right side, by showing that each $u\in \Delta$ is a root of $H(Y):=\sum_{i=1}^n Y^{\frac{q^{i-1}-1}{q-1}}$. Pick $u\in \Delta$, 
so that $u = (w^q-w)^{q-1}$ with $w\in \F_{q^n}\setminus\F_q$, and thus $H(u)=\sum_{i=1}^n (w^q-w)^{q^{i-1}-1}$. Then
\[
(w^q-w)H(u)=\sum_{i=1}^n (w^q-w)^{q^{i-1}}
=\sum_{i=1}^n \bigl(w^{q^i}-w^{q^{i-1}}\bigr) = w^{q^n}-w=0,
\]
so since $w\notin\F_q$ we conclude that $H(u)=0$.
\end{proof}

We now prove Theorem~\ref{identity}.

\begin{proof} [Proof of Theorem~\ref{identity}]
Since each side of \eqref{apeq} is a polynomial in $\F_{q^n}[X,Y]$ having $X$-degree less than $q^n-1$, it suffices to 
show that, for each of the $q^n-1$ values $u\in\F_{q^n}^*$, we obtain an equality in $\F_{q^n}[Y]$ when we substitute 
$u$ for $X$ in \eqref{apeq}. Pick any $u\in\F_{q^n}^*$, and put $Z:=u^{-q}Y$. Then the left side of \eqref{apeq} 
becomes $\prod_{z\in\F_{q^n}^*}(1+zu-z^q u^q Z)$; upon substituting $w:=zu$, it follows from Lemma~\ref{1st} and 
\ref{2nd} that this equals
\begin{equation}\label{fu}
- \biggl( \frac{ Z^{q^n} - Z } { Z^{\frac{q^n-1}{q-1}} - 1 } \biggr) \cdot \sum_{i=1}^n Z^{\frac{q^i-q}{q-1}}.
\end{equation}
After our substitutions, the right side of \eqref{apeq} becomes
\[
-u^q Z\cdot \biggl( \frac{ 1 - Z^{q^n-1} } { u^{\frac{q^n-1}{q-1}} - (u^q Z)^{\frac{q^n-1}{q-1}} } \biggr) \cdot 
\Bigl((u^q Z)^{\frac{q^n-q}{q-1}}+\sum_{i=1}^{n-1} u^{1+\frac{q^n-q^{i+1}}{q-1}}(u^q Z)^{\frac{q^i-q}{q-1}}\Bigr).
\]
Since $u^{(q^n-1)/(q-1)}$ is in $\F_q^*$, it equals its $q$-th power, so the above expression becomes
\[
-u^q Z\cdot \biggl( \frac{ 1 - Z^{q^n-1} } { u^{\frac{q^n-1}{q-1}} - (u Z)^{\frac{q^n-1}{q-1}} } \biggr) \cdot 
\Bigl(u^{\frac{q^n-1}{q-1}-q} Z^{\frac{q^n-q}{q-1}}+\sum_{i=1}^{n-1} u^{1+\frac{q^n-q^2}{q-1}} Z^{\frac{q^i-q}{q-1}}\Bigr),
\]
which equals \eqref{fu}.
\end{proof}


\section{Resolution of eight conjectures and open problems}
\label{sec:conj}

In this section we resolve eight conjectures and open problems from the literature. The following result generalizes 
the combination of \cite[Open Question~1]{LXZ} and \cite[Thm.~1(1)]{LXZ}:

\begin{cor}\label{pre1}
Put $q:=2^k$ and $Q:=2^\ell$ where $k$ and $\ell$ are odd positive integers, and write $A(X) := aX^{Q+1}+bX^Q+cX+d$ with 
$a,b,c,d \in \F_{q^2}$. Then $f(X) := X^{Q+1} A(X^{q-1})$ permutes\/ $\F_{q^2}$ if and only if all of the following hold:
\renewcommand{\theenumi}{\ref{pre1}.\arabic{enumi}}
\renewcommand{\labelenumi}{\emph{(\thethm.\arabic{enumi})}}
\begin{enumerate}
\item\label{11} $e := a^{q+1}+b^{q+1}+c^{q+1}+d^{q+1}$ is nonzero;
\item\label{12} $(a^q b + c^q d)^Q = e^{Q-1} (a^q c + b^q d)$;
\item\label{13} $\Tr_{\F_q/\F_2}\Bigl(\frac{b^{q+1}+c^{q+1}}e\Bigr) = 1$.
\end{enumerate}
\end{cor}
\renewcommand{\theenumi}{\arabic{enumi}}
\renewcommand{\labelenumi}{(\arabic{enumi})}

\begin{proof}
Since $k$ is odd, Lemma~\ref{gcd} implies that $\gcd(Q+1,q-1)=1$. By Theorem~\ref{fir}, $f(X)$ permutes $\F_{q^2}$ 
if and only if \eqref{11} and \eqref{13} hold and the $q$-th power of \eqref{12} holds. Plainly \eqref{12} holds 
if and only if its $q$-th power holds, so the result follows.
\end{proof}

Next we show that the sufficient conditions in \cite[Thm.~1.1]{ZLKPT} are both necessary and sufficient, as was 
conjectured in \cite[pp.~5 and 20]{ZLKPT}:

\begin{cor}\label{pre2}
Let $k$ and $\ell$ be positive integers and write $m:=2^{\ord_2(\gcd(k,\ell))}$. Suppose that $q:=2^k$ and $Q:=2^\ell$ 
satisfy $\gcd(Q-1,q+1)=1$. Write $A(X):=1+bX+aX^v+dX^u$ where $a,b,d\in\F_{q^2}$ and $u,v$ are positive integers 
such that $u(Q-1)\equiv q\pmod{q+1}$ and $v(Q-1)\equiv Q\pmod{q+1}$. Then $f(X):=X A(X^{q-1})$ permutes\/ $\F_{q^2}$ 
if and only if all of the following hold:
\renewcommand{\theenumi}{\ref{pre2}.\arabic{enumi}}
\renewcommand{\labelenumi}{\emph{(\thethm.\arabic{enumi})}}
\begin{enumerate}
\item\label{21} $e:=1+a^{q+1}+b^{q+1}+d^{q+1}$ is nonzero;
\item\label{22} $(ab^q+d^q)^Q=e^{Q-1}(a+bd^q)$;
\item\label{23} $\Tr_{\F_q/\F_{2^m}}\Bigl(\frac{a^{q+1}+d^{q+1}}e\Bigr)=0$.
\end{enumerate}
\end{cor}
\renewcommand{\theenumi}{\arabic{enumi}}
\renewcommand{\labelenumi}{(\arabic{enumi})}

\begin{proof}
We first restate \eqref{23}. The hypothesis $\gcd(Q-1,q+1)=1$ says $\ord_2(k) \ge \ord_2(\ell)$ by 
Lemma~\ref{gcd}, so that $k/m\equiv \lcm(k,\ell)/m\pmod{2}$. Since $\Tr_{\F_q/\F_{2^m}}(1)=k/m=\lcm(k,\ell)/m$ 
and $1+(a^{q+1}+d^{q+1})/e=(b^{q+1}+1)/e$, it follows that \eqref{23} holds if and only if 
$\Tr_{\F_q/\F_{2^m}}((b^{q+1}+1)/e)=\lcm(k,\ell)/m$. Let $s$ be a positive integer such that 
$s\equiv Q-1\pmod{q+1}$ and $\gcd(s,q-1)=1$, so that $\gcd(s,q^2-1)=1$. Then $f(X)$ permutes $\F_{q^2}$ 
if and only if $f(X^s)$ permutes $\F_{q^2}$. Since
\[
A(X^s)\equiv 1+bX^{Q-1}+aX^Q+dX^q\pmod{X^{q+1}+1},
\]
we see that $f(X^s)$ induces the same function on $\F_{q^2}$ as does $f_1(X):=X^r A_1(X^{q-1})$, where 
$r:=s+q^2-q$ and $A_1(X):=aX^{Q+1}+bX^Q+X+d$. Note that $r\equiv s\pmod{q-1}$ so that $\gcd(r,q-1)=1$, 
and also $r\equiv Q+1\pmod{q+1}$. By Theorem~\ref{fir}, $f_1(X)$ permutes $\F_{q^2}$ if and only if 
\eqref{21} holds, \eqref{22} holds, and our restatement of \eqref{23} holds.
\end{proof}

Next we show that the sufficient conditions in \cite[Thm.~1.3]{ZLKPT} are both necessary and sufficient, 
as was conjectured in \cite[pp.~5 and 20]{ZLKPT}:

\begin{cor}\label{pre3}
Let $k$ and $\ell$ be positive integers and write $m:=2^{\ord_2(\gcd(k,\ell))}$. Suppose that $q:=2^k$ and 
$Q:=2^\ell$ satisfy $\gcd(Q+1,q+1)=1$. Write $A(X):=1+aX+bX^v+cX^u$ where $a,b,c\in\F_{q^2}$ and $u,v$ 
are positive integers such that $u(Q+1)\equiv 1\pmod{q+1}$ and $v(Q+1)\equiv Q\pmod{q+1}$. Then 
$f(X):=X A(X^{q-1})$ permutes\/ $\F_{q^2}$ if and only if all of the following hold:
\renewcommand{\theenumi}{\ref{pre3}.\arabic{enumi}}
\renewcommand{\labelenumi}{\emph{(\thethm.\arabic{enumi})}}
\begin{enumerate}
\item\label{31} $e:=1+a^{q+1}+b^{q+1}+c^{q+1}$ is nonzero;
\item\label{32} $(a^qb+c^q)^Q=e^{Q-1} (a^qc + b^q)$;
\item\label{33} $\Tr_{\F_q/\F_{2^m}}\Bigl(\frac{b^{q+1}+c^{q+1}}e\Bigr)=0$.
\end{enumerate}
\end{cor}
\renewcommand{\theenumi}{\arabic{enumi}}
\renewcommand{\labelenumi}{(\arabic{enumi})}

\begin{proof}
The hypothesis $\gcd(Q+1,q+1)=1$ says $\ord_2(k) \ne \ord_2(\ell)$ by Lemma~\ref{gcd}, so that 
$\lcm(k,\ell)/m\equiv 0\pmod{2}$. Let $r$ be a positive integer such that $r\equiv Q+1\pmod{q+1}$ and 
$\gcd(r,q-1)=1$, so that $\gcd(r,q^2-1)=1$. Thus $f(X)$ permutes $\F_{q^2}$ if and only if $f(X^r)$ 
permutes $\F_{q^2}$. Since
\[
A(X^r)\equiv aX^{Q+1}+bX^Q+cX+1\pmod{X^{q+1}+1},
\]
we see that $f(X^r)$ induces the same function on $\F_{q^2}$ as does $f_1(X):=X^r A_1(X^{q-1})$, 
where $A_1(X):=aX^{Q+1}+bX^Q+cX+1$. By Theorem~\ref{fir}, $f_1(X)$ permutes $\F_{q^2}$ if and only if 
\eqref{31}, \eqref{33}, and the $q$-th power of \eqref{32} all hold (and of course \eqref{32} is 
equivalent to its $q$-th power).
\end{proof}

The next result generalizes the combination of \cite[Open problem 2]{ZKPT} and \cite[Thm.~3.3]{ZKPT}:

\begin{cor}\label{pre4}
Let $k$ and $\ell$ be positive integers and write $m:=2^{\ord_2(\gcd(k,\ell))}$. Suppose that $q:=2^k$ and 
$Q:=2^\ell$ satisfy $\gcd(Q+1,q+1)=1$. Write $A(X):=1+bX^v+cX^u$ where $b,c\in\F_{q^2}$ and $u,v$ are positive 
integers such that $u(Q+1)\equiv 1\pmod{q+1}$ and $v(Q+1)\equiv Q\pmod{q+1}$. Then $f(X):=X A(X^{q-1})$ 
permutes $\F_{q^2}$ if and only if all of the following hold:
\renewcommand{\theenumi}{\ref{pre4}.\arabic{enumi}}
\renewcommand{\labelenumi}{\emph{(\thethm.\arabic{enumi})}}
\begin{enumerate}
\item\label{41} $e:=1+b^{q+1}+c^{q+1}$ is nonzero;
\item\label{42} $c^Q=e^{Q-1} b$;
\item\label{43} $\Tr_{\F_q/\F_{2^m}}\bigl(\frac1e\bigr)=\frac{k}{m}$.
\end{enumerate}
\end{cor}
\renewcommand{\theenumi}{\arabic{enumi}}
\renewcommand{\labelenumi}{(\arabic{enumi})}

\begin{proof}
We first rewrite \eqref{43}. Since $\Tr_{\F_q/\F_{2^m}}(1)=k/m$ and $1+1/e=(b^{q+1}+c^{q+1})/e$, we see that \eqref{43} 
holds if and only if $\Tr_{\F_q/\F_{2^m}}((b^{q+1}+c^{q+1})/e)=0$. Now the case $a=0$ of Corollary~\ref{pre3} says that 
$f(X)$ permutes $\F_{q^2}$ if and only if \eqref{41} and \eqref{43} hold and the $q$-th power of \eqref{42} holds 
(or equivalently \eqref{42} holds).
\end{proof}

The next result shows that the sufficient conditions in \cite[Thm.~1]{LHXZ} are both necessary and sufficient, 
as was conjectured in \cite[Rem.~2]{LHXZ}:

\begin{cor}\label{pre5}
Let $k$ and $\ell$ be coprime positive integers with $k$ odd, and write $q:=2^k$ and $Q:=2^\ell$. 
Write $R(X,Y):=(X+\alpha Y)^{Q+1}+(\beta Y)^{Q+1}$ with $\alpha,\beta\in\F_q^*$. Then the function 
$\psi\colon (x,y)\mapsto \bigl(R(x,y),R(y,x)\bigr)$ induces a permutation of\/ $\F_q\times\F_q$ 
if and only if $\alpha^2+\alpha\beta+\beta^2=1$.
\end{cor}

The first part of our proof of Corollary~\ref{pre5} also shows that the sufficient conditions in \cite[Thm.~1]{LLHQ} 
are both necessary and sufficient. This implies the following refinement of \cite[Conj.~19]{LLHQ}, in which the 
hypothesis on boomerang uniformity in this conjecture has been removed:

\begin{cor}\label{pre5b}
Under the hypotheses and notation of Corollary~\emph{\ref{pre5}}, the function $\psi$ permutes\/ $\F_q\times\F_q$ 
if and only if $\beta\ne\alpha+1$ and $u^Q=e^{Q-1}v$, where
\begin{align*}
e &:= (\alpha+1)^{2Q+2}+\beta^{2Q+2}, \\
u &:= (\alpha+1)^{2Q+2}+\alpha^{2Q+1}+\alpha+\alpha^Q\beta^{Q+1}+\beta^{2Q+2}, \\
v &:= (\alpha+1)^{2Q+2}+\alpha^{Q+2}+\alpha^Q+\alpha\beta^{Q+1}+\beta^{2Q+2}.
\end{align*}
\end{cor}

\begin{rmk}
In Corollary~\ref{pre5b} we write $\beta^{Q+1}$ for the quantity which is called $\beta$ in \cite{LLHQ}. This does 
not cause any loss of generality, since the hypothesis that $k$ is odd implies that $\gcd(q-1,Q+1)=1$, so that the 
$(Q+1)$-th power map permutes $\F_q$.
\end{rmk}

We will prove Corollaries~\ref{pre5} and \ref{pre5b} simultaneously. Our proof yields an alternate proof 
of \cite[Thm.~2(1)]{LHXZ2}. Conversely, our proof of Corollary~\ref{pre5} could be shortened a bit by using 
\cite[Thm.~2(1)]{LHXZ2}. We chose to keep our current proof because it involves the unexpected connection with 
value sets of non-permutation polynomials from the previous section.

\begin{proof}[Proof of Corollaries~\ref{pre5} and \ref{pre5b}]
We first reduce to the case that $\ell$ is odd. Writing $\tilde{\ell}:=k+\ell$ and $\widetilde{Q}:=2^{\tilde{\ell}}$, 
if we replace $\ell$ by $\tilde{\ell}$ in the corollaries then we replace $R(X,Y)$ by 
$\widetilde{R}(X,Y):=(X+\alpha Y)^{\widetilde{Q}+1}+(\beta Y)^{\widetilde{Q}+1}$. Since $R(X,Y)$ and $\widetilde{R}(X,Y)$ 
induce the same function on $\F_q\times\F_q$, it follows that $\psi$ permutes $\F_q\times\F_q$ if and only if 
$\widetilde{\psi}\colon (x,y)\mapsto \bigl(R(x,y)^{\widetilde{Q}},R(y,x)^{\widetilde{Q}}\bigr)$ does. Thus the corollaries 
are true for $\ell$ if and only if they are true when we replace $\ell$ by $\tilde{\ell}$. If $\ell$ is even then 
$\tilde{\ell}$ is odd, so we may assume in what follows that $\ell$ is odd.

Pick an order-$3$ element $\omega\in\F_{q^2}$. Since $k$ is odd, the elements $1$ and $\omega$ form a basis 
for $\F_{q^2}$ as an $\F_q$-vector base. Thus the map $\varphi\colon (x,y)\mapsto x+\omega y$ is a bijection 
$\varphi\colon\F_q\times\F_q\to\F_{q^2}$, so that $\psi$ permutes $\F_q\times\F_q$ if and only if 
$\widehat{\psi}:=\varphi\circ\psi\circ\varphi^{-1}$ permutes $\F_{q^2}$. Writing any $z\in\F_{q^2}$ as $x+\omega y$ 
with $x,y\in\F_q$, we have $z^q=x+\omega^2 y$ so that $z+z^q=y$ and $\omega^2 z+\omega z^q=x$. Thus we have
\begin{align*}
\widehat{\psi}(z)&= (x+\alpha y)^{Q+1}+(\beta y)^{Q+1}+\omega(y+\alpha x)^{Q+1}+\omega(\beta x)^{Q+1} \\
&= (1+\omega \alpha^{Q+1}+\omega \beta^{Q+1})x^{Q+1} + (\alpha+\omega \alpha^Q)x^Q y + (\alpha^Q+\omega \alpha)x y^Q \\
&\qquad + (\alpha^{Q+1}+\beta^{Q+1}+\omega) y^{Q+1},
\end{align*}
which equals $\omega^2 az^{qQ+q}+\omega bz^{qQ+1}+cz^{q+Q}+\omega^2 dz^{Q+1}$ where 
\begin{align*}
a & :=\alpha^{Q+1}+\beta^{Q+1}+\alpha^Q+1, \\
b & :=\alpha^{Q+1}+\beta^{Q+1}+\alpha^Q+\alpha, \\
c & :=\alpha^Q+\alpha+1, \\
d & :=\alpha^{Q+1}+\beta^{Q+1}+\alpha+1.
\end{align*}
%
%
For $z\in\F_{q^2}$ we have $\omega\widehat{\psi}(\omega^2 z)=f(z)$ where 
\[
f(X):=aX^{qQ+q}+bX^{qQ+1}+cX^{q+Q}+dX^{Q+1}.
\] 
Since $\omega X$ and $\omega^2 X$ permute $\F_{q^2}$, it follows that $\widehat{\psi}$ permutes $\F_{q^2}$ if 
and only if $f(X)$ permutes $\F_{q^2}$. Since $k$ is odd, $r:=Q+1$ is coprime to $q-1$. Now the case $r=Q+1$ 
of Theorem~\ref{fir} says $f(X)$ permutes $\F_{q^2}$ if and only if all of the following hold:
\stepcounter{equation}
\renewcommand{\theenumi}{\thethm.\arabic{enumi}}
\renewcommand{\labelenumi}{(\thethm.\arabic{enumi})}
\begin{enumerate}
\item\label{51} $\widetilde{e}:=a^2+b^2+c^2+d^2$ is nonzero;
\item\label{52} $(ab+cd)^Q=\widetilde{e}^{Q-1}(ac+bd)$;
\item\label{53} $\Tr_{\F_q/\F_2}\bigl( (b^2+c^2)/\widetilde{e} \bigr)=1$.
\end{enumerate}
\renewcommand{\theenumi}{\arabic{enumi}}
\renewcommand{\labelenumi}{(\arabic{enumi})}
We compute $\widetilde{e} = ((\alpha+1)^{Q+1}+\beta^{Q+1})^2 = e$,
%
%
so that \eqref{51} says $(\alpha+1)^{Q+1}\ne \beta^{Q+1}$. Since $\gcd(q-1,Q+1)=1$, this is equivalent 
to $\alpha+1\ne \beta$. One can check that \eqref{52} is equivalent to the condition $u^Q=e^{Q-1}v$ 
in Corollary~\ref{pre5b}. Next, if \eqref{51} and \eqref{52} hold then a routine computation yields
\[
\frac{b^2+c^2}e = \lambda^Q + \lambda + 1
\]
where
\[
\lambda := \frac{u+e}e\alpha.
\]
%
%
Since $\Tr_{\F_q/\F_2}(\lambda^Q+\lambda+1)=1$, it follows that \eqref{51} and \eqref{52} imply \eqref{53}, 
which concludes the proof of Corollary~\ref{pre5b}.

Next, \eqref{52} says $H(\alpha,\beta)=0$ where
\begin{align*}
H(X,Y):=&X^{2Q^2+2Q+1}Y + X^{2Q^2+Q}Y^{Q+2} + X^{2Q^2+1}Y + X^{Q^2+2Q+2}Y^{Q^2} \\ &+ X^{Q^2+2Q}Y^{Q^2}
+
    X^{Q^2+2}Y^{Q^2} + X^{Q^2}Y^{Q^2+2Q+2} + X^{Q^2}Y^{Q^2} \\ &+ X^{2Q+1}Y + X^{Q+2}Y^{2Q^2+Q} + 
    X^QY^{2Q^2+Q} + X^QY^{Q+2} \\ &+ XY^{2Q^2+2Q+1} + XY.
\end{align*}

The special case of \eqref{apeq} in which $q$ and $n$ are replaced by $Q$ and $2$ says
\begin{equation*}
 \prod_{z \in \F_{Q^2}^*} ( 1 + z X + z^Q Y )  =  1 + X^{Q^2-1} +
 Y  \Big( \frac{ X^{Q^2-1} + Y^{Q^2-1} } { X^{Q+1} + Y^{Q+1} } \Big) 
  \Big( Y^Q + X \Big) .
\end{equation*}
Homogenize this by substituting $X/Z$ and $Y/Z$ for $X$ and $Y$, and then multiplying by $Z^{Q^2-1}$, to get
\begin{align*}
 \prod_{z \in \F_{Q^2}^*} ( Z + z X + z^Q Y ) &=  Z^{Q^2-1} + X^{Q^2-1} \\[-0.6cm]
 &\quad + Y  \Big( \frac{ X^{Q^2-1} + Y^{Q^2-1} } { X^{Q+1} + Y^{Q+1} } \Big) 
  \Big( Y^Q + XZ^{Q-1} \Big) .
\end{align*}
Multiplying both sides by $Z(X^{Q+1}+Y^{Q+1})$ yields
\begin{align*}
Z&\cdot(X^{Q+1}+Y^{Q+1})\cdot\prod_{z \in \F_{Q^2}^*} ( Z + z X + z^Q Y )\\ &\qquad\qquad=  (Z^{Q^2} + X^{Q^2-1}Z)\cdot (X^{Q+1}+Y^{Q+1})\\
 &\qquad\qquad\qquad + (X^{Q^2-1}Y + Y^{Q^2}) \cdot
 ( Y^Q Z + XZ^Q ) .
\end{align*}
Now substitute $\alpha^2+1$ for $X$, $\beta^2$ for $Y$, and $\alpha\beta$ for $Z$, in order to obtain
\begin{align*}
\alpha\beta&\Bigl((\alpha+1)^{Q+1}+\beta^{Q+1}\Bigr)^2\prod_{z \in \F_{Q^2}^*} \Bigl( \alpha\beta + z (\alpha+1)^2 + z^Q \beta^2 \Bigr)\\ 
&= \Bigl((\alpha\beta)^{Q^2} + (\alpha+1)^{2Q^2-2}\alpha\beta\Bigr)\cdot \Bigl((\alpha+1)^{Q+1}+\beta^{Q+1}\Bigr)^2 \\
 &\qquad + \Bigl((\alpha+1)^{Q^2-1}\beta + \beta^{Q^2}\Bigr)^2 \cdot \Bigl( \alpha\beta^{2Q+1} + (\alpha+1)^2 (\alpha\beta)^Q \Bigr).
\end{align*}
The right side of this last equation equals $H(\alpha,\beta)$.
%
%
Since \eqref{52} says $H(\alpha,\beta)=0$, and $\alpha,\beta\ne 0$ by hypothesis, it follows that \eqref{51} and \eqref{52} 
are both true if and only if $(\alpha+1)^{Q+1}\ne \beta^{Q+1}$ and $\alpha\beta+z(\alpha+1)^2+z^Q\beta^2=0$ for some $z\in\F_{Q^2}^*$.

We now show that if $\alpha+1\ne \beta$ and $z\in\F_{Q^2}\setminus\F_2$ then $\alpha\beta+z(\alpha+1)^2+z^Q\beta^2\ne 0$. 
Suppose otherwise. Since $\alpha\beta\in\F_q$, we have $\alpha\beta=(\alpha\beta)^q$, so that
\[
z(\alpha+1)^2+z^Q\beta^2 = z^q (\alpha+1)^2 + z^{qQ} \beta^2,
\]
or equivalently
\[
(z+z^q)(\alpha+1)^2 = (z+z^q)^Q \beta^2.
\]
Since $z\in\F_{Q^2}\setminus\F_2$ and $\F_{Q^2}\cap\F_q=\F_2$, we have $z\notin\F_q$, so that $z+z^q\ne 0$ and thus
\[
\Bigl(\frac{\alpha+1}\beta\Bigr)^2=(z+z^q)^{Q-1}.
\]
Here the left side is in $\F_q$ and the right side is in $\F_{Q^2}^*$, so since $\F_q\cap\F_{Q^2}=\F_2$ 
it follows that both sides equal $1$. But then $\alpha+1=\beta$, contradicting our hypothesis.

Thus \eqref{51} and \eqref{52} both hold if and only if $(\alpha+1)^{Q+1}\ne \beta^{Q+1}$ and $\alpha\beta+(\alpha+1)^2+\beta^2=0$. 
In fact the latter condition implies the former: if $\alpha\beta+(\alpha+1)^2+\beta^2=0$ then $\alpha+1\ne \beta$ since 
$\alpha\beta\ne 0$, which implies $(\alpha+1)^{Q+1}\ne \beta^{Q+1}$ since $X^{Q+1}$ permutes $\F_q$. Since we showed above 
that \eqref{51} and \eqref{52} imply \eqref{53}, this concludes the proof of Corollary~\ref{pre5}.
%
%
%
%
\end{proof}

Our next result refines the sufficient condition in \cite[Thm.~3.2]{ZKP} to obtain a necessary and sufficient 
condition, as discussed in \cite[Rem.~2]{ZKP}:

\begin{cor}\label{pre6}
Let $k$ and $\ell$ be positive integers and write $m:=2^{\ord_2(\gcd(k,\ell))}$. Suppose that $q:=2^k$ and 
$Q:=2^\ell$ satisfy $\gcd(Q-1,q+1)=1$. Write $A(X):=1+aX^u+dX^v$ where $a,d\in\F_q$ and $u,v$ are positive 
integers such that $u(Q-1)\equiv Q\pmod{q+1}$ and $v(Q-1)\equiv q\pmod{q+1}$. Then $f(X):=X A(X^{q-1})$ 
permutes\/ $\F_{q^2}$ if and only if all of the following hold:
\renewcommand{\theenumi}{\ref{pre6}.\arabic{enumi}}
\renewcommand{\labelenumi}{\emph{(\thethm.\arabic{enumi})}}
\begin{enumerate}
\item\label{61} $e:=1+a^2+d^2$ is nonzero;
\item\label{62} $d^Q=e^{Q-1}a$;
\item\label{63} $\Tr_{\F_q/\F_{2^m}}\Bigl(\frac{a^2+d^2}e\Bigr)=0$.
\end{enumerate}
\end{cor}
\renewcommand{\theenumi}{\arabic{enumi}}
\renewcommand{\labelenumi}{(\arabic{enumi})}

\begin{proof}
Let $s$ be a positive integer such that $\gcd(s,q-1)=1$ and $s\equiv Q-1\pmod{q+1}$. Then $\gcd(s,q^2-1)=1$, 
so that $X^s$ permutes $\F_{q^2}$, and thus $f(X)$ permutes $\F_{q^2}$ if and only if $f(X^s)$ permutes $\F_{q^2}$. 
We now determine the reductions mod $q^2-1$ of the degrees of the terms of 
\[
f(X^s)=X^s+aX^{su(q-1)+s}+dX^{sv(q-1)+s}.
\]
Here $r:=sv(q-1)+s$ is congruent to $s\pmod{q-1}$, and hence is coprime to $q-1$, while also 
\[
r\equiv (Q-1)v(q-1)+Q-1\equiv q(q-1)+Q-1\equiv Q+1\pmod{q+1}.
\]
Next, 
\[
su(q-1)-sv(q-1)\equiv (Q-q)(q-1)\equiv (Q+1)(q-1)\pmod{q+1},
\]
so that $su(q-1)-sv(q-1)\equiv (Q+1)(q-1)\pmod{q^2-1}$ since both sides are divisible by $q-1$. Likewise, 
we have $-sv(q-1)\equiv q-1\pmod{q^2-1}$. Thus $f(X^s)$ induces the same function on $\F_{q^2}$ as does
\[
f_1(X):=X^{r+q-1}+aX^{r+(Q+1)(q-1)}+dX^r.
\]
By Theorem~\ref{fir}, $f_1(X)$ permutes $\F_{q^2}$ if and only if \eqref{61}, \eqref{62}, and the following 
equation \eqref{63b} all hold:
\begin{equation}\label{63b}
\Tr_{\F_q/\F_{2^m}}\Bigl(\frac1e\Bigr)=\frac{\lcm(k,\ell)}m.
\end{equation}
It remains to show that \eqref{63b} is equivalent to \eqref{63}. We compute
\[
\Tr_{\F_q/\F_{2^m}}\Bigl(\frac1e\Bigr) =\Tr_{\F_q/\F_{2^m}}\Bigl(1+\frac{a^2+d^2}e\Bigr)=\frac{k}{m}+\Tr_{\F_q/\F_{2^m}}\Bigl(\frac{a^2+d^2}e\Bigr).
\]
The hypothesis $\gcd(Q-1,q+1)=1$ implies that $\ord_2(k) \ge \ord_2(\ell)$, so that $\lcm(k,\ell)/m \equiv k/m \pmod{2}$.  
Thus indeed \eqref{63b} holds if and only if \eqref{63} holds.
\end{proof}

\begin{rmk} 
Theorem~3.2 of \cite{ZKP} asserts that the polynomial $f(X)$ in Corollary~\ref{pre6} permutes $\F_{q^2}$ 
if certain conditions hold. These sufficient conditions include \eqref{61}--\eqref{63} in addition to 
other conditions, the most notable of which is $(a+d)\Tr_{\F_q/\F_2}(d/e)=0$. Corollary~\ref{pre6} shows 
that this last condition is superfluous. In fact, it can be deduced from the combination of Lemma~\ref{u|A} 
and Theorem~\ref{main} that if $\ord_2(k)\ne\ord_2(\ell)$ then \eqref{61}--\eqref{63} imply that 
$\Tr_{\F_q/\F_2}(d/e)=0$, but if $\ord_2(k)=\ord_2(\ell)$ then the permutation polynomials in 
Corollary~\ref{pre6} for which $(a+d)\Tr_{\F_q/\F_2}(d/e)\ne 0$ are
\[
X + \frac{\lambda^{Q-2}+\lambda^Q}{1+\lambda^{2Q-2}}X^{u(q-1)+1} + \frac{1+\lambda^{2Q}}{\lambda+\lambda^{2Q-1}}X^{v(q-1)+1}
\]
where $\lambda\in\mu_{q+1}\setminus\mu_{Q+1}$. In particular, such examples exist whenever $Q$ is not a power 
of $q$. Moreover, one can show that distinct $\lambda,\lambda'\in\mu_{q+1}\setminus\mu_{Q+1}$ yield the same 
polynomial if and only if $\lambda'=1/\lambda$.
\end{rmk}

Our final result resolves \cite[Open problem]{ZKP} by showing that the sufficient conditions in \cite[Thm.~3.1]{ZKP} 
are also necessary:

\begin{cor}\label{pre7}
Let $k$ and $\ell$ be positive integers with $\gcd(2k,\ell)=1$, and write $q:=2^k$ and $Q:=2^\ell$ and $u:=(Q^{k+1}-1)/(Q-1)$. 
Pick $a\in\F_{q^2}^*$ and $b\in\F_q^*$ with $a+b\ne 1$. Then $f(X):=X+bX^q+aX^u$ permutes\/ $\F_{q^2}$ if and only if there 
exist $\lambda\in\F_q$ and $\epsilon\in\F_2$ such that all of the following hold:
\renewcommand{\theenumi}{\ref{pre7}.\arabic{enumi}}
\renewcommand{\labelenumi}{\emph{(\thethm.\arabic{enumi})}}
\begin{enumerate}[leftmargin=\parindent,align=left,labelwidth=6pt,labelsep=6pt,itemsep=2pt]
\item\label{71} $\lambda^{Q-1}=b$;
\item\label{72} $a=\epsilon\lambda^Q+\sum_{i=1}^{\ell} \lambda^{Q-2^i}$;
\item\label{73} $\Tr_{\F_q/\F_2}(\frac{a}{a+b+1})=0$.
\end{enumerate}
\end{cor}
\renewcommand{\theenumi}{\arabic{enumi}}
\renewcommand{\labelenumi}{(\arabic{enumi})}

\begin{proof}
The hypothesis $\gcd(2k,\ell)=1$ implies that $\gcd(q^2-1,Q-1)=1$, so that $X^{Q-1}$ permutes $\F_{q^2}$.  
Thus $f(X)$ permutes $\F_{q^2}$ if and only if $f(X^{Q-1})$ permutes $\F_{q^2}$. Note that
\[
u(Q-1)-(Q-q)=(Q^{k+1}-1)-(Q-q)=Q(Q^k-1)+(q-1)=Q(q^\ell-1)+(q-1)
\]
is divisible by $q-1$, and
\[
Q(q^\ell-1)+(q-1)\equiv -2Q-2\pmod{q+1},
\]
so that
\[
u(Q-1)-(Q-q)\equiv (Q+1)(q-1)\pmod{q^2-1}.
\]
Thus if $r$ is a positive integer satisfying $r\equiv Q-q\pmod{q^2-1}$ then $f(X^{Q-1})$ induces the same function 
on $\F_{q^2}$ as does 
\[
f_1(X):=X^{r+q-1}+bX^{r+Q(q-1)}+aX^{r+(Q+1)(q-1)}.
\] 
Note that $r\equiv Q+1\pmod{q+1}$ and $r\equiv Q-1\pmod{q-1}$, so that $\gcd(r,q-1)=1$. Thus Theorem~\ref{fir} 
implies that $f_1(X)$ permutes $\F_{q^2}$ if and only if all of the following hold:
\stepcounter{equation}
\renewcommand{\theenumi}{\theequation.\arabic{enumi}}
\renewcommand{\labelenumi}{(\thethm.\arabic{enumi})}
\begin{enumerate}[leftmargin=\parindent,align=left,labelwidth=6pt,labelsep=6pt,itemsep=2pt]
\item\label{71b} $e:=a^{q+1}+b^2+1$ is nonzero;
\item\label{72b} $(ab)^Q=e^{Q-1}a$;
\item\label{73b} $\Tr_{\F_q/\F_2}\bigl(\frac{b^2+1}e\bigr)=k$.
\end{enumerate}
\renewcommand{\theenumi}{\arabic{enumi}}
\renewcommand{\labelenumi}{(\arabic{enumi})}
It remains to show that \eqref{71b}--\eqref{73b} are equivalent to \eqref{71}--\eqref{73}.

First suppose that \eqref{71b}--\eqref{73b} hold. Since $a\ne 0$ by hypothesis, \eqref{72b} says $a^{Q-1}b^Q=e^{Q-1}$.  
Since $b$ and $e$ are in $\F_q^*$, it follows that $a^{Q-1}\in\F_q^*$, which since $\gcd(Q-1,q^2-1)=1$ implies that 
$a\in\F_q^*$. Thus $e=a^2+b^2+1=(a+b+1)^2$, so that
\[
\frac{b^2+1}e=1+\frac{a^2}e=1+\Bigl(\frac{a}{a+b+1}\Bigr)^2.
\]
Since $\eta:=a/(a+b+1)$ is in $\F_q$, we have $\Tr_{\F_q/\F_2}(\eta^2)=\Tr_{\F_q/\F_2}(\eta)$, so since 
$\Tr_{\F_q/\F_2}(1)=k$ it follows that
\[
\Tr_{\F_q/\F_2}\Bigl(\frac{b^2+1}e\Bigr) = k + \Tr_{\F_q/\F_2}\Bigl(\frac{a}{a+b+1}\Bigr).
\]
Thus \eqref{73b} implies that $\Tr_{\F_q/\F_2}(a/(a+b+1))=0$, which is \eqref{73}. Writing $\lambda$ for the unique $(Q-1)$-th 
root of $b$ in $\F_q$ (so that \eqref{71} holds), the identity $a^{Q-1}b^Q=e^{Q-1}$ becomes $(a\lambda^Q)^{Q-1}=e^{Q-1}$. Since 
$X^{Q-1}$ permutes $\F_q$, it follows that $a\lambda^Q=e$. Thus
\[
\lambda^{2Q-2} = b^2 = e+a^2+1 = a\lambda^Q+a^2+1,
\]
and dividing by $\lambda^{2Q}$ yields
\[
\frac1{\gamma^2} = \frac{a}{\lambda^Q} + \frac{a^2}{\lambda^{2Q}}+\frac1{\lambda^{2Q}}.
\]
This says $\epsilon^2+\epsilon=0$ where
\[
\epsilon:=\frac{a}{\lambda^Q} + \sum_{i=1}^\ell\frac1{\lambda^{2^i}},
\]
so that $\epsilon$ is in $\F_2$. Multiplying by $\lambda^Q$ yields
\[
a=\epsilon\lambda^Q+\sum_{i=1}^\ell \lambda^{Q-2^i},
\]
which is \eqref{72}. We have shown that if \eqref{71b}--\eqref{73b} hold then \eqref{71}--\eqref{73} hold.

Conversely, now suppose that \eqref{71}--\eqref{73} hold. Condition \eqref{72} implies $a\in\F_q$. Dividing \eqref{72} by 
$\lambda^Q$ yields
\[
\frac{a}{\lambda^Q} = \epsilon + \sum_{i=1}^{\ell} \frac1{\lambda^{2^i}}. 
\]
Add each side of this equation to its square to get
\[
\frac{a}{\lambda^Q} + \frac{a^2}{\lambda^{2Q}} = \frac{1}{\lambda^2} + \frac{1}{\lambda^{2Q}}.
\]
Now multiply by $\lambda^{2Q}$ to conclude that
\[
a\lambda^Q + a^2 = \lambda^{2Q-2} + 1,
\]
or equivalently
\begin{equation}\label{redgreen}
ab\lambda=a^2+b^2+1.
\end{equation}
Since $ab\lambda \ne 0$ and $a\in \F_q$, it follows that \eqref{71b} holds. Raising both sides of \eqref{redgreen} 
to the $(Q-1)$-th power yields $a^{Q-1}b^Q=e^{Q-1}$, so that \eqref{72b} holds. Finally, we have $(b^2+1)/e = 1 + a^2/e$. 
Since $\Tr_{\F_q/\F_2}(1)=k$ and $\Tr_{\F_q/\F_2}(a^2/e)=\Tr_{\F_q/\F_2}(a/(a+b+1))=0$ by \eqref{73}, we conclude that
\[
\Tr_{\F_q/\F_2}\Bigl(\frac{b^2+1}e\Bigr) = \Tr_{\F_q/\F_2}\Bigl(1+\frac{a^2}e\Bigr) = k,
\]
so that \eqref{73b} holds.
\end{proof}

\begin{rmk}
Our proofs of the above corollaries show that all permutation polynomials in the corollaries are multiplicatively 
equivalent to the permutation polynomials in Theorem~\ref{fir}. Some of the above corollaries can be generalized 
to larger classes of permutation polynomials which are also multiplicatively equivalent to the polynomials in 
Theorem~\ref{fir}. For instance, Corollary~\ref{pre7} can be generalized as follows: if $q=2^k$ and $Q=2^\ell$ 
where $\gcd(2k,\ell)=1$, and $f(X):=aX^u+bX^q+cX+dX^v$ where $a,b,c,d\in\F_{q^2}$ and $u:=(Q^{k+1}-1)/(Q-1)$ and 
$v:=1 + q(Q^k-1)/(Q-1)$, then $f(X)$ permutes $\F_{q^2}$ if and only if $a,b,c,d$ satisfy the conditions in 
Theorem~\ref{fir}. We have not stated this as a separate result, or listed any further results along these lines, 
since this result provides no new understanding and instead is merely an immediate consequence of Theorem~\ref{fir}. 
\end{rmk}



\end{document}